\documentclass[a4paper,12pt]{amsart}

\usepackage[latin1]{inputenc}
\usepackage{amsfonts}
\usepackage{amsmath}
\usepackage{amssymb}
\usepackage{amsthm}
\usepackage[english]{babel}
\usepackage{mathdots}
\usepackage[all]{xy}
\usepackage{hyperref, url}
\usepackage{pinlabel}
\usepackage{color}

\usepackage[margin=3cm]{geometry}

\newtheorem{thm}{Theorem}[section]
\newtheorem{lemma}[thm]{Lemma}
\newtheorem{prop}[thm]{Proposition}
\newtheorem{cor}[thm]{Corollary}
\newtheorem{defn}[thm]{Definition}

\newtheorem{ques}[thm]{Question}
\newtheorem{claim}[thm]{Claim}

\theoremstyle{remark}
\newtheorem{ex}[thm]{Example}
\newtheorem{rmk}[thm]{Remark}

\numberwithin{figure}{section}
\numberwithin{equation}{section}

\newcommand{\N}{\mathbb{N}}
\newcommand{\Z}{\mathbb{Z}}
\newcommand{\Q}{\mathbb{Q}}
\newcommand{\R}{\mathbb{R}}
\newcommand{\C}{\mathbb{C}}
\newcommand{\CP}{\mathbb{CP}}
\newcommand{\F}{\mathbb{F}}

\renewcommand{\O}{\mathcal{O}}
\newcommand{\de}{\partial}
\newcommand{\HFoo}{HF^\infty}
\newcommand{\CFKoo}{CFK^\infty}
\newcommand{\s}{\mathfrak{s}}
\newcommand{\spinc}{spin$^c$~}
\renewcommand{\t}{\mathfrak{t}}
\renewcommand{\epsilon}{\varepsilon}
\DeclareMathOperator{\Int}{Int}
\DeclareMathOperator{\im}{im}
\DeclareMathOperator{\Spinc}{Spin^c}
\DeclareMathOperator{\PD}{PD}

\DeclareSymbolFont{extraup}{U}{zavm}{m}{n}
\DeclareMathSymbol{\varspadesuit}{\mathalpha}{extraup}{85}

\newcommand{\e}{\emph}

\title{Cuspidal curves and Heegaard Floer homology}
\author{J\'ozsef Bodn\'ar, Daniele Celoria, Marco Golla}
\date{}

\begin{document}

\begin{abstract}
We give bounds on the gap functions of the singularities of a cuspidal plane curve of arbitrary genus, generalising recent work of Borodzik and Livingston. We apply these inequalities to unicuspidal curves whose singularity has one Puiseux pair: we prove two identities tying the parameters of the singularity, the genus, and the degree of the curve; we improve on some degree-multiplicity asymptotic inequalities; finally, we prove some finiteness results, we construct infinite families of examples, and in some cases we give an almost complete classification.
\end{abstract}

\maketitle

\section{Introduction}

A curve $C\subset \CP^2$ is cuspidal if each singular point admits a neighbourhood $U$ such that the intersection of the curve $C$ and the boundary of $U$ is connected. That is to say, this intersection, called the link of the singularity, is a knot.

Our focus will be on unicuspidal curves, that is, cuspidal curves with only one singular point. Recently, Borodzik and Livingston studied the rational case, i.e.~when the resolution of $C$ is a sphere, and obtained a strong constraint on some coefficients of the Alexander polynomials of the link of the singularity~\cite{BL}, proving a conjecture of Fern\'andez de Bobadilla, Luengo, Melle-Hernandez and N\'emethi~\cite{BLMN}.

We extend their result to prove an analogous result for arbitrary unicuspidal curves.

\begin{thm}\label{thm:main}
Suppose $C$ is a cuspidal curve of degree $d$ and genus $g$ with one singular point; let $I$ be the gap function associated to the singularity. Then for every $-1\le j\le d-2$ and every $0\le k \le g$ integers we have:
\[
k-g \le I_{jd+1-2k} - \frac{(d-j-2)(d-j-1)}2  \le k.
\]
\end{thm}

There is a generalisation for curves with more cusps, stated as Theorem~\ref{t:morecusps} below.

We then turn to the numerical study of 1-\e{unicuspidal} curves, i.e.~curves with one cusp singularity that has only one Puiseux pair (that is, its link is a torus knot rather than an iterated torus knot).

\begin{thm}\label{thm:pell}

Fix a positive integer $g \geq 1$. Let $C$ be a 1-unicuspidal curve of genus $g$ and degree $d$ whose singularity is of type $(a,b)$. Then, if $d$ is sufficiently large,
\begin{equation}\label{eq:deg}
 a + b = 3d
\end{equation}
or, equivalently,
\begin{equation}\label{eq:pell}
\left(\frac{7b-2a}{3}\right)^2 - 5b^2 = 4(2g-1).
\end{equation}
\end{thm}

Theorem~\ref{thm:pell} imposes strong restrictions on pairs $(a,b)$ that can be realised as Puiseux pairs of the singularity of a plane unicuspidal curve of genus $g$, in the spirit of~\cite{BLMN1}. In particular, we obtain the following two corollaries.

\begin{cor}\label{c:finite}
For all genera $g$ with $g\equiv 2 \pmod 5$ or $g\equiv 4 \pmod 5$ there are only finitely many 1-unicuspidal, genus-$g$ curves up to equisingularity.
\end{cor}

The second corollary is a degree-multiplicity inequality in the spirit of Matsuoka--Sakai~\cite{MS} and Orevkov~\cite{Or}. For convenience let $\phi$ denote the golden ratio $\phi = \frac{1+\sqrt5}2$.

\begin{cor}\label{c:acc}
Let $g\ge 1$. Then there exists a constant $c$ such that
\[ \phi^2 a - c < d < \phi^2 a + c \]
for every 1-unicuspidal genus-g curve of degree $d$ with Puiseux pair $(a,b)$.
\end{cor}

\begin{rmk}
The case $g=0$ is excluded in Theorem~\ref{thm:pell}: singularities of 1-unicuspidal rational curves have been classified in~\cite{BLMN1}, and the result does not hold in this case. However, applying Theorem~\ref{thm:main} (which, as pointed out above, is the main theorem of~\cite{BL}) we can recover the four infinite families of singularities obtained in~\cite{BLMN1} (see Remarks~\ref{rmk:g01stprop} and~\ref{rmk:g02ndprop}).

The proof of Theorem~\ref{thm:pell} relies almost exclusively on Theorem~\ref{thm:main}, except when $g = 1$. In that case, Theorem~\ref{thm:main} alone cannot exclude the family $(a,b) = (l, 9l+1)$ with $d = 3l$ (Case VII in the proof of Proposition~\ref{p:67}). This family can be excluded using an inequality due to Orevkov~\cite{Or}, as pointed out by Borodzik, Hedden and Livingston~\cite{BHL} (see also Remark~\ref{rem:parallel} below).
\end{rmk}

\begin{rmk} \label{rem:openq1}
Corollary~\ref{c:acc} in particular shows that an analogue of Orevkov's asymptotic inequality between the multiplicity and the degree  holds true for any fixed genus $g$ in the special case of 1-unicuspidal curves. Even more surprisingly, in the 1-unicuspidal case, for any fixed genus $g\ge1$ an asymptotic inequality \emph{in the opposite direction} holds as well.
\end{rmk}

Finally, we construct an infinite family of 1-unicuspidal curves for each triangular genus. We set up some notation first. Given an integer $k$, denote with $(L^k_n)_{n\in\Z}$ the Lucas sequence defined by the data $L^k_0 = k-1, L^k_1 = 1$ and the recurrence $L^k_{n+1} = L^k_n + L^k_{n-1}$. Notice that $n$ varies among integers rather than positive integers.

It is easy to check that for every $i \geq 2$ the pair $(a, b) = (L_{4i-3}^{k}, L_{4i+1}^{k})$ is a solution of~\eqref{eq:pell} if $g = k(k-1)/2$. In this case, the degree is $d = L^k_{4i-1}$. Also, for $j \geq 1$ the pair $(a, b) = (-L_{-4j+1}^{k}, -L_{-4j-3}^{k})$ is a solution of~\eqref{eq:pell} if $g = k(k-1)/2$. In this case, the degree is $d = -L^k_{-4j-1}$.

\begin{thm}\label{t:construction}
Let $k\ge2$ be an integer, and define $g = k(k-1)/2$. For each $i \geq 2$ there exists a unicuspidal projective plane curve of genus $g$ and degree $d = L_{4i-1}^{k}$ such that the singularity has one Puiseux pair $(a, b) = (L_{4i-3}^{k}, L_{4i+1}^{k})$. Similarly, for each $j \geq 1$ there exists a unicuspidal projective plane curve of genus $g$ and degree $d = -L_{-4j-1}^{k}$ such that the singularity has one Puiseux pair $(a, b) = (-L_{-4j+1}^{k}, -L_{-4j-3}^{k})$.

Moreover, if $k\equiv 2\pmod 3$ and $2g-1$ is a power of a prime, then any $1$-unicuspidal curve of genus $g$ and sufficiently large degree has one of the singularities listed above.
\end{thm}

\begin{rmk}\label{rem:parallel}
Upon finishing this manuscript, we learned that Theorem~\ref{thm:main} was independently proved by Maciej Borodzik, Matthew Hedden and Charles Livingston~\cite{BHL}. The applications they have, however, are different: they classify singular curves of genus 1 having degree larger than 33 and one cusp with one Puiseux pair.
\end{rmk}

\subsection*{Organisation of the paper}
In Section~\ref{s:plane-curves} we recall some notation regarding complex plane curves and their singularities; in Section~\ref{s:top} we review the topological setup, and in Section~\ref{s:HF} we review some necessary background in Heegaard Floer homology and work out some auxiliary computations.
Section~\ref{s:mainproof} is devoted to the proof of Theorem~\ref{thm:main}, while Theorem~\ref{thm:pell}, Corollaries~\ref{c:finite} and~\ref{c:acc} are proved in Section~\ref{s:acc}.
Section~\ref{s:pell} is a short trip in number theory, where we study the solutions of Equation~\eqref{eq:pell}. Finally, we prove Theorem~\ref{t:construction} and give some examples in Section~\ref{s:examples}.

\subsection*{Acknowledgements}
We would like to thank: the referee, for many helpful comments and suggestions; Paolo Lisca, Andr\'as N\'emethi, and Andr\'as Stipsicz for interesting conversations; Gabriele Dalla Torre for useful comments on Section~\ref{s:pell}. The first author is supported by ERC grant LDTBud at MTA Alfr\'ed R\'enyi Institute of Mathematics. The second author is supported by the PhD school of the University of Florence. The third author has been supported by the EU Advanced Grant LDTBud, a CAST Exchange Grant, the PRIN--MIUR research project 2010--11 ``Variet\`a reali e complesse: geometria, topologia e analisi armonica'', and the FIRB research project ``Topologia e geometria di variet\`a in bassa dimensione''.

\section{Plane curves}\label{s:plane-curves}
The discussion in this section closely follows~\cite[Section 2]{BL}; we also refer the reader to the classical books~\cite{Milnor, Wall} for further information.

Recall that the zero-set of a nonzero homogeneous polynomial $f\in\C[x,y,z]$ of degree $d$ gives rise to a \emph{plane curve} $C=V(f)\subset\CP^2$; we call $d$ the \e{degree} of the curve $C$.
The set where the gradient of $f$ vanishes along $C$ is called the \emph{singular set}, which is a discrete set provided $f$ does not have multiple components.

Consider a small ball $B$ centered at a singular point $p\in C$. The link of $C$ at $p$ is the isotopy class of the intersection $\de B \cap C$, that is the isotopy class of a link in $S^3$. We say that $p$ is a \emph{cusp} if such intersection is connected. In this case, the intersection is a \e{nontrivial} knot $K$, and the \e{Milnor number} $2\delta$ of the singularity is twice the Seifert genus of $K$, i.e.~$\delta = g(K)$.

A curve is called \emph{cuspidal} if all of its singular points are cusps. A cuspidal curve is homeomorphic to a topological surface of genus $g$. 
Recall that the degree-genus formula yields $\sum_i\delta_i + g = \frac{(d-1)(d-2)}2$ (see, for example, \cite[Section II.11]{Barth}).

Given a singular point $p$, we consider the set $\Gamma\subset\Z$ defined as follows: $\Gamma$ is the set of local multiplicities of intersections of germs of complex curves with $C$ at $p$. It is easy to see that $\Gamma$ is closed under addition and contains $0$, and is called the \e{semigroup} of the singularity. We denote with $G = \Z\setminus\Gamma$ the set of \e{gaps} of the semigroup $\Gamma$.

Associated to the semigroup $\Gamma$ are the \e{semigroup counting function} $R\colon\Z\to \Z$ and the \e{gap counting function} $I\colon\Z\to \Z$, defined by
\begin{align*}
R_m &= \#(\Gamma\cap(-\infty,m-1]),
\\
I_m &= \#(G\cap[m,+\infty)).
\end{align*}
For example, for every singularity we have that $R_1 = 1$ and $I_0=\delta$. Moreover, it is always the case that $\max G = 2\delta-1$, so that $I_m = 0$ if $m\ge 2\delta$ and $R_{2\delta} = \delta$.

Later we will also use the notation $I(m)$ instead of $I_m$ for convenience.

Recall from~\cite[Lemma 6.2]{BL} that we have
\begin{equation}\label{e:RI}
 R_m = m - \delta + I_m,
\end{equation}
as a corollary of the symmetry property of the semigroup.

Every germ of a curve singularity can be parametrised, in an appropriate chart, by a function $t\mapsto (t^a, t^{b_1}+\dots+t^{b_m})$, where $1<a<b_1<\dots<b_m$ are positive integers such that $\gcd(a,b_1,\dots,b_k)$ does not divide $b_{k+1}$ for every $0\le k \le m-1$ and $\gcd(a,b_1,\dots,b_m) = 1$. We will say that the singularity has one \e{Puiseux pair} if $m = 1$, and we will say that the singular point is of \e{type} $(a,b)$.

\begin{ex}
When the singularity is of type $(a,b)$ the link of the singularity is isotopic to a torus knot $T(a,b)$. The semigroup of the singularity in this case is generated by $a$ and $b$: these are the multiplicities of intersection of $C$ with the coordinate planes in the chart given above, where $C$ is defined by the equation $y^a-x^b=0$. Accordingly, $\delta = (a-1)(b-1)/2$.
\end{ex} 
\section{Topology}\label{s:top}

Let $C\subset \CP^2$ be a cuspidal curve of degree $d$ and genus $g$, and let $p_1,\dots,p_n$ be its singular points. 
We want to give a handle decomposition of a regular neighbourhood $N_C$ of $C$ and a description of the algebraic topology of its complement $-W_C$, i.e.~$W_C = -(\CP^2 \setminus \Int N_C)$.
This will in turn give a surgery description for the boundary $Y_C$ of $W_C$: $Y_C = \de N_C = \de W_C$.

For each $i$ fix a small 4-ball neighourhood $B_i$ of $p_i$ in $\CP^2$. The intersection $\de B_i \cap C$ is isotopic to the link $K_i$ of the singular point $p_i$; denote with $2\delta_i$ its Milnor number.

Now fix a regular 3-ball neighbourhood $D_i\subset\de B_i$ of a point of $K_i$ for each $i$, so that $D_i\cap K_i$ is an unknotted arc in $D_i$. Fix another regular 3-ball neighbourhood $D\subset \de B_1\setminus D_1$ of a point of $K_1$ that intersects $K_1$ in an unknotted arc. Finally, fix a topological handle decomposition of $C$ with the following properties (see Figure~\ref{f:handles}):

\begin{itemize}
\item there are only $n$ 0-handles $B_i\cap C$ and one 2-handle;
\item there are $2g$ 1-handles, whose feet $q_1, q'_1, \dots, q_{2g}, q'_{2g}$ land in $K_1\cap D$;
\item the order of the points $q_i, q'_i$ along the arc $K_1\cap D$ is $q_1,q_2,q'_1,q'_2,\dots$ $q_{2g-1},q_{2g},q'_{2g-1},q'_{2g}$;
\item there are $n-1$ 1-handles whose feet land in the union of the discs $D_i$.
\end{itemize}

\begin{figure}[h]
\labellist
\small\hair 2pt
\pinlabel $B_1$ at 0 50
\pinlabel $D$ at -5 125
\pinlabel $D_1$ at 70 150
\pinlabel $B_2$ at 220 0
\pinlabel $D_2$ at 110 180
\pinlabel $B_3$ at 220 100
\pinlabel $D_3$ at 110 35
\endlabellist
\centering
\includegraphics[scale=.9]{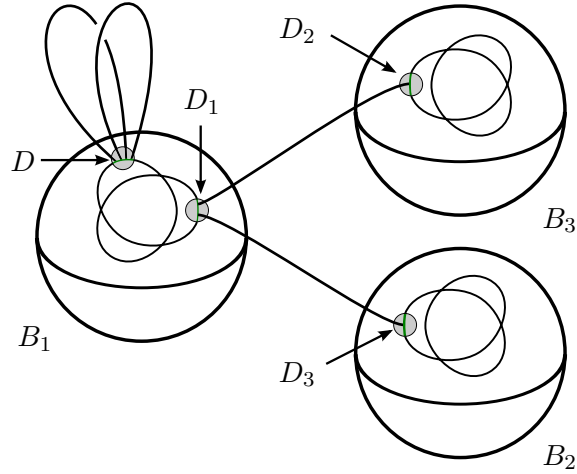}
\caption{A schematic picture of the 0- and 1-handles of the handle decompositions of $C$ and $U$ of Section~\ref{s:top}. It represents the case of a genus-1 curve with three cusps.}\label{f:handles}
\end{figure}

Now, fix a regular neighbourhood $U$ of the union of the balls $B_i$ and the cores of the 1-handles of the decomposition in $\CP^2$. Notice that $U$ is a 4-dimensional 1-handlebody, and therefore $\de U$ is diffeomorphic to $\#^{2g} S^1\times S^2$.

Denote with $K_C$ the connected sum of the knots $K_i$, $K_C = K_1 \# \dots \# K_n$, and with $\delta$ its Seifert genus, which is equal to $\sum \delta_i$.

\begin{lemma}
The 4-manifold $N_C$ is obtained from $U$ by attaching a single 2-handle along the connected sum of $K_C\subset S^3$ and the Borromean knot $K_B$ in $\#^{2g}S^1\times S^2$ (described below), with framing $d^2$.
\end{lemma}

The \e{Borromean knot} $K_B$ in $\#^{2g} S^1\times S^2$ is the boundary of the surface $F\times\{*\}$ inside (a smoothing of) the boundary of $F\times D^2$, where $F$ is the compact, once punctured surface of genus $g$ and $*$ is a point on $\de D^2$.
It is described by the Kirby diagram of Figure~\ref{f:borromean}.

\begin{figure}[h!]
\labellist
\small\hair 2pt
\pinlabel $\underbrace{\hphantom{------------------}}_g$ at 230 45
\pinlabel $K_B$ at -20 20
\endlabellist
\centering
\includegraphics[scale=.5]{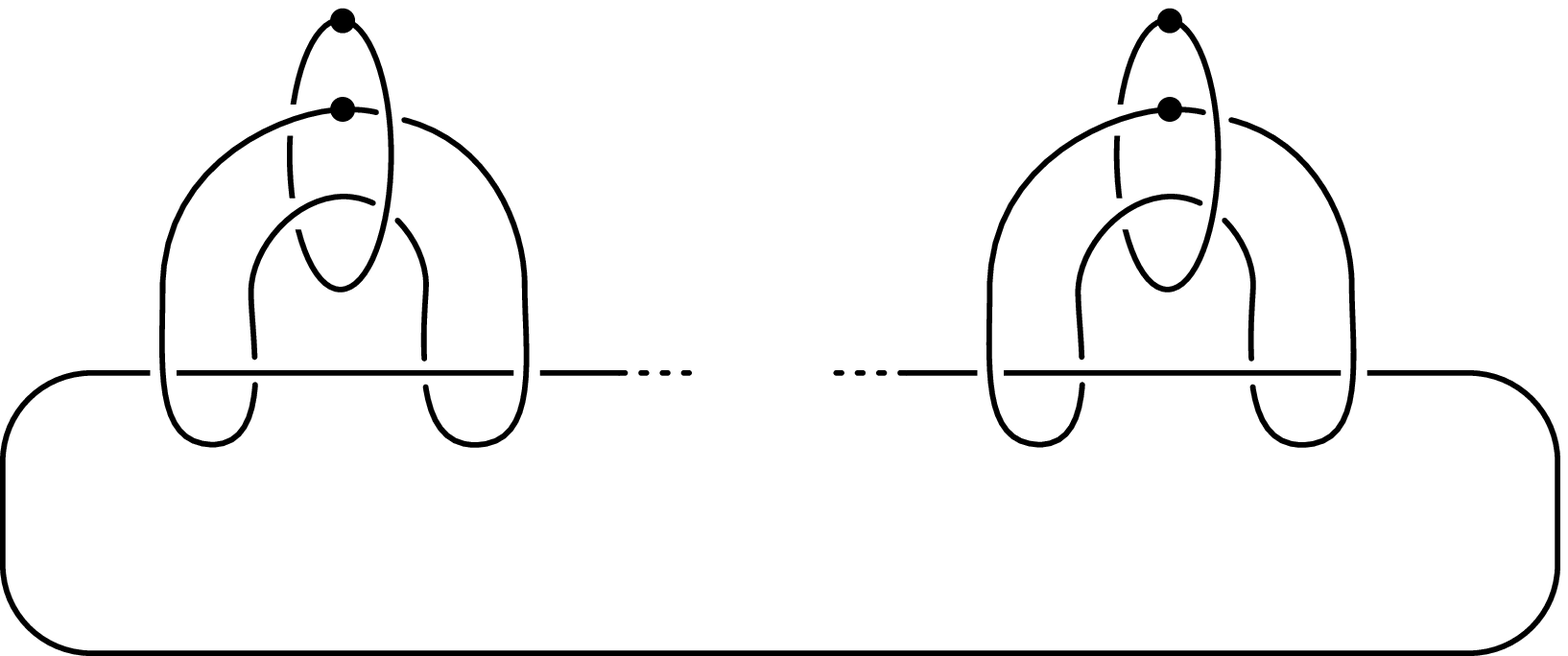}
\caption{A Kirby diagram for the Borromean knot in $\#^{2g}S^1\times S^2$.}\label{f:borromean}
\end{figure}

\begin{proof}
Call $S, S_1,\dots S_n$ the boundaries of the 3-balls $D, D_1, \dots, D_n$ respectively; these are $n+1$ 2-spheres in the 3-manifold $\de U$. $S, S_2, \dots, S_n$ separate $\de U$ into several pieces: $n$ 3-balls $\de B_i \setminus D_i$; a punctured $\#^{2g}S^1\times S^2$; a piece that retracts onto $\de B_1\setminus (D\cup D_1)$, which is homeomorphic to $S^2\times [0,1]$. Moreover, $C\cap \de U$ intersects each of the spheres $S, S_1,\dots S_n$ in two points, therefore exhibiting a decomposition of the attaching curve of the 2-handle as a connected sums of knots $K_1,\dots K_n$ in $\de B_1,\dots, \de B_n$ and a knot $K'$ in $\#^{2g}S^1\times S^2$.

By definition, each knot $K_i$ is the link of the singularity of $C$ at $p_i$. The knot $K'$ is easily seen to be the Borromean knot. In fact, $K'$ is the boundary of $C\setminus B$, which is a once punctured surface of genus $g$, and a neighbourhood of $C\setminus B$ inside $\CP^2$ is diffeomorphic to $(C\setminus B) \times D^2$.

The framing is determined by the self-intersection of a surface which is homologous to $C$, hence it is $d^2$.
\end{proof}

\begin{lemma}\label{l:homology}
The 3-manifold $Y_C$ has $H_1(Y_C;\Z) = \Z/d^2\Z\oplus \Z^{2g}$. The 4-manifold $W_C$ has $b_2^{\pm}(W_C) = 0$, $b_2(W_C) = 2g$ and $H_1(W_C;\Z) = \Z/d\Z$.
\end{lemma}

In what follows, we will write $H_*(-)$ instead of $H_*(-;\Z)$.

\begin{proof}
The 3-manifold $Y_C$ is obtained as $d^2$-surgery along a nullhomologous knot in $\#^{2g} S^1\times S^2$, therefore its first homology is $\Z/d^2\Z\oplus \Z^{2g}$.

Let us now apply the Mayer--Vietoris long exact sequence to the decomposition of $\CP^2$ into $N_C$ and $-W_C$.

Since $N_C$ retracts onto $C$, $H_*(N_C) = H_*(C)$ and the map $H_2(N_C)\to H_2(\CP^2)$ induced by the inclusion is multiplication by $d$, where we identify $H_2(N_C)$ with $\Z$ by mapping $[C]$ to 1, and we identify $H_2(\CP^2)$ with $\Z$ by mapping the hyperplane class $H$ to 1.

Also, notice that every smooth representative of a class in $H_2(W_C)$ is disjoint from $C$, therefore it intersects $H$ trivially: it follows that the inclusion $W_C\hookrightarrow \CP^2$ induces the trivial map $H_2(W_C)\to H_2(\CP^2)$ and that $b_2^\pm(W_C) = 0$.

From the Mayer--Vietoris long exact sequence we get the following exact sequences:
\begin{align*}
\begin{split}
&0\to H_4(\CP^2) \to H_3(Y_C) \to H_3(N_C)\oplus H_3(W_C)\to H_3(\CP^2) = 0\\
&0\to H_2(Y_C) \to H_2(N_C)\oplus H_2(W_C) \to H_2(\CP^2) \to L \to 0\\
&0\to L \to H_1(Y_C) \to H_1(N_C)\oplus H_1(W_C)\to H_1(\CP^2) = 0
\end{split}
\end{align*}
where $L$ is the cokernel of the map $H_2(N_C)\oplus H_2(W_C) \to H_2(\CP^2)$. From the first line we deduce that $H_3(W_C) = 0$.

The remarks on the maps $H_2(N_C)\to H_2(\CP^2)$ and $H_2(W_C)\to H_2(\CP^2)$ show that the inclusion $Y_C\hookrightarrow W_C$ induces an isomorphism $H_2(Y_C)\to H_2(W_C)$ and that $L = \Z/d\Z$. In particular, $b_2(W_C) = 2g$. More precisely, we can see that $H_2(W_C)$ is generated by embedded tori (of square 0, called \e{rim tori}), each of which projects onto an essential curve in $C$ under the projection $N_C\to C$.

Finally, the map $H_1(Y_C)\to H_1(N_C)$ restricts to an isomorphism on the free part of $H_1(Y_C)$; let $T\subset H_1(Y_C)$ be the torsion subgroup; we have another exact sequence:
\[
0\to L \to T \to H_1(W_C) \to 0
\]
from which the last claim follows.
\end{proof}

The following lemma deals with the extension of spin$^c$ structures from $Y_C$ to $W_C$. The proof of~\cite[Lemma 3.1]{BL} applies \e{verbatim} here, and we refer the reader to it.

\begin{lemma}
The spin$^c$ structure $\s_m$ on $Y_C$ extends to $W_C$ if and only if $m = kd$, where $k\in [-d/2,d/2]$ is an integer if $d$ is odd, and a half-integer if $d$ is even.
\end{lemma}

\section{Heegaard Floer homology}\label{s:HF}

Let us consider a closed, oriented spin$^c$ 3-manifold $(Y,\t)$ such that $c_1(\t)$ is torsion in $H^2(Y;\Z)$: we call this a \e{torsion} spin$^c$ 3-manifold. Let $H$ denote $H_1(Y;\Z)$ modulo its torsion subgroup: Ozsv\'ath and Szab\'o associate to such a pair two $\Q$-graded $\bigwedge^*H\otimes\F[U]$-modules, $\HFoo({Y},\t)$ and $HF^+({Y},\t)$, called the \e{Heegaard Floer homology groups} of $({Y},\t)$~\cite{OS3I, OSPA, OSAG}. Here $\F$ is the field with two elements $\F = \Z/2\Z$.

The pair $({Y},\t)$ is said to have \e{standard} $\HFoo$ if
\[
\HFoo({Y},\t) = \textstyle{\bigwedge}^* H^1(Y;\Z) \otimes \F[U,U^{-1}]
\]
where the action by $H$ on the right-hand side is given by contraction.

There is a canonical map $\pi\colon\HFoo({Y},\t)\to HF^+({Y},\t)$, and we say that an element of $HF^+(Y,\t)$ is \e{nontorsion} if it is in the image of $\pi$. The map $\pi$ allows us to associate two numbers to a torsion spin$^c$ 3-manifold with standard $\HFoo$.

\begin{defn}
The \e{correction term} $d({Y},\t)$ of a torsion spin$^c$ 3-manifold $({Y},\t)$ is the minimal degree of an element in $\im(\pi)$.

The \e{bottom-most correction term} $d_b({Y},\t)$ of $({Y},\t)$ is the minimal degree of an element in $\pi(K_H)$, where $K_H$ is the kernel of the action by $H$ on $\HFoo(Y,\t)$.
\end{defn}

The following theorem is due to Ozsv\'ath and Szab\'o.

\begin{thm}[{\cite[Theorem 9.15]{OSAG}}]\label{d-inequality}
If $(W,\s)$ is a spin$^c$ 4-manifold with boundary $(Y,\t)$, then under the following hypotheses:
\begin{itemize}
 \item $W$ is negative semidefinite;
 \item the restriction map $H^1(W;\Z)\to H^1(Y;\Z)$ is trivial;
 \item $c_1(\t)$ is torsion;
 \item $(Y,\t)$ has standard $\HFoo$;
\end{itemize}
the following inequality holds:
\begin{equation}\label{eqn:d-ineq}
 c_1^2(\s) + b_2^-(W) \le 4d_b(Y,\t) + 2b_1(Y).
\end{equation}
\end{thm}

We want to apply the inequality above to $W_C$ and its boundary $Y_C$ from the previous section.

\begin{lemma}\label{l:standardHFoo}
$Y_C$ and $-Y_C$ have standard $\HFoo$ in their torsion spin$^c$ structures.
\end{lemma}

\begin{proof}
The 3-manifold $-Y_C$ is obtained from $\#^{2g}S^1\times S^2$, which has standard $\HFoo$ in its torsion spin$^c$ structure, by a negative surgery along the nullhomologous knot $K_C\#K_B$.

Hence~\cite[Proposition 9.4]{OSAG} shows that $-Y_C$ has standard $\HFoo$, and~\cite[Proposition 2.5]{OSPA} implies that $Y_C$ has standard $\HFoo$, too.
\end{proof}

Since $\HFoo(Y_C,\t)$ is standard for any torsion spin$^c$ structure $\t$, the bottom-most correction terms $d_b(Y_C,\t)$ are defined. We index the torsion spin$^c$ structures on $Y_C$ by elements in $\Z/d^2\Z$ in the following way (see also~\cite[Section 2]{OSIS}).

Let $Z$ denote the surgery cobordism from $\#^{2g}S^1\times S^2$ to $Y_C$, and denote with $A$ the class of $H_2(Z)$ represented by a Seifert surface for $K_C\#K_B$, capped off with the core of the 2-handle. For every integer $i$ there is a unique \spinc structure $\s_i$ on $Z$ such that:
\begin{itemize}
\item the restriction of $\s_i$ to $\de_-Z = \#^{2g}S^1\times S^2$ is the unique torsion \spinc structure on $\#^{2g}S^1\times S^2$;
\item $\langle c_1(\s_i), A\rangle = d^2+2i$.
\end{itemize}
We denote with $\t_i$ the restriction of $\s_i$ to $\de_+Z = Y_C$, and note that $\t_i = \t_{i+d^2}$, therefore the $\Z$-labelling descends to a labelling of $\Spinc(Y_C)$ by $\Z/d^2\Z$ (see~\cite[Section 2.4]{OSIS}).

We introduce some more notation. Given a knot $K$ in $S^3$, we associate to it a family of integers $\{V_m(K)\}_{m\ge0}$ defined as follows~\cite{NW, Jake}. For every fixed $m\ge0$, the quantity
\[
d(S^3_N(K), \t_m) + \frac{N-(2m- N)^2}{4N}
\]
is independent of $N$ when $N>2g(K)$, and we define $V_m(K)$ as
\[
V_m(K) = -\frac12\left(d(S^3_N(K), \t_m) + \frac{N-(2m- N)^2}{4N}\right)
\]
for any such $N$. Here the labelling of spin$^c$ structures on $S^3_N(K)$ is the one induced by the surgery cobordisms along $K$, as above. While we do not need it in what follows, we remark here that the summand $-\frac{N-(2m- N)^2}{4N}$ is the correction term of the lens space $L(N,-1)$ in the \spinc structure $\t_m$, with the labelling induced by the surgery cobordism associated to $+N$-surgery along the unknot in $S^3$, as above.

\begin{prop}\label{p:dbottom}
Using the notation and the constants from the previous section, we can compute the bottom-most correction terms of $\pm Y_C$:
\begin{align*}
d_b(Y_C,\t_m) &= \min_{0\le k \le g}\{2k-g-2V_{m+g-2k}(K_C)\} - d_-(m),\\
d_b(-Y_C,\t_m) &= \min_{0\le k \le g}\{2m-2k+g+2V_{m+g-2k}(K_C)\} + d_+(m),
\end{align*}
where $\displaystyle d_\pm(m) = \frac{d^2-(2m\pm d^2)^2}{4d^2}$.
\end{prop}
Let us recall some definitions and results from~\cite[Section 4]{OSKI}.

Consider a nullhomologous knot $K$ in a closed 3-manifold $Y$, and let $C(K)=\CFKoo(K)$ denote its knot Floer homology complex.

For every integer $m$, we have two quotients of $C(K)$:
\begin{itemize}
\item The subset $C(K)\{\max(i,j-m)<0\}$ is a subcomplex of $C(K)$, and we denote the corresponding quotient by $C^b_m(K)$, and the quotient map $C(K) \to C^b_m(K)$ by $\pi^b_m$; $C^b_m(K)$ is called the \e{big} complex;
\item the subset $C(K)\{\min(i,j-m)<0\}$ is a subcomplex of $C(K)$, and we denote the corresponding quotient by $C^s_m(K)$, and the quotient map $C(K) \to C^s_m(K)$ by $\pi^s_m$; $C^s_m(K)$ is called the \e{small} complex.
\end{itemize}

Graphically, the big complex corresponds to the three ``non-negative'' quadrants with vertex in $(0,m)$, and the small complex corresponds to the ``positive'' 
quadrant with vertex in $(0,m)$.

It is proven in~\cite{JakePhD} and~\cite[Section 4]{OSKI} that, for every $N>2g(K)$, the complexes $C^b_m(K)$ and $C^s_m(K)$ compute the Floer homology of $(Y_N(K), \t_m)$ and $(Y_{-N}(K), \t_m)$ respectively, with a degree shift (see in particular \cite[Corollary 4.2 and Remark 4.3]{OSKI}). Namely,
\begin{align}
HF^+_{*-\frac{N-(2m - N)^2}{4N}}(Y_N(K),\t_m) \simeq H_*(C^b_m(K))\label{e:psurg}\\
HF^+_{*+\frac{N-(2m + N)^2}{4N}}(Y_{-N}(K),\t_m) \simeq H_*(C^s_m(K))\label{e:nsurg}
\end{align}
Moreover, this isomorphism respects the action of 
$H_1(Y; \Z)/{\rm Torsion} = H_1(Y_N(K);\Z)/{\rm Torsion}$: this is implicit in the statement of \cite[Theorem 4.4]{OSKI}, and it follows from the equivariance of cobordism maps associated to nonzero surgeries along nullhomologous knots under the $H_1$-action.

We pause here to state and prove the following lemma. In what follows, given a knot $K$ we denote with $m(K)$ its mirror. We also denote with $2V'_m(K)$ the minimal degree of an element in $H_*(C^s_m(K))$ that is nontorsion.

\begin{lemma}\label{lemma:Hm}
For every $m\ge 0$,  $V'_m(m(K)) = V_m(K)+m$.
\end{lemma}

\begin{proof}
Given an integer $m$ and a sufficiently large $N$, Equation~\eqref{e:nsurg} asserts that the complex $C^s_m(m(K))$ computes the Heegaard Floer homology of $S^3_{-N}(m(K))$, with a degree shift by $\frac{N-(2m + N)^2}{4N}$.

More precisely, the minimal degree $2V'_m(m(K))$ of a nontorsion element in $H_*(C^s_m(m(K)))$ is $d(S^3_{-d^2}(m(K)),\t_m)-\frac{N-(2m + N)^2}{4N}$.
On the other hand, correction terms of rational homology spheres change sign when reversing the orientation~\cite[Proposition 4.2]{OSAG}, so we have
\[
-2V_m(K) - \textstyle{\frac{N-(2m - N)^2}{4N}} = d(S^3_{N}(K),\t_m) = -d(S^3_{-N}(m(K)),\t_m) = -2V'_m(m(K)) - \textstyle{\frac{N-(2m + N)^2}{4N}},
\]
from which the lemma follows.
\end{proof}

We are going to apply the surgery formulae to $(Y,K) = (S^3,K_C)\#(\#^{2g} S^1\times S^2, K_B)$ and its mirror; indeed, the 3-manifold $Y_C$ is obtained as $+d^2$-surgery along $K$, while $-Y_C$ is obtained as $-d^2$-surgery along $m(K)$.

Also, the knot Floer homology of $K_B$ as an $H_1(\#^{2g} S^1\times S^2;\Z)$-module has been computed in~\cite[Proposition 9.2]{OSKI}. 

Identify $H_1(\#^{2g} S^1\times S^2;\Z)$ with $H_1(\Sigma;\Z)$, where $\Sigma$ is a closed genus-$g$ surface. Endow the module
\[
C(K_B) = \Z[U,U^{-1}] \otimes_\Z {\textstyle{\bigwedge}}^* H^1(\Sigma; \Z)
\]
with the trivial differential; define the Alexander and Maslov gradings on $C(K_B)$ so that the summand $U^{-i} \otimes \Lambda^{g-i+j} H^1(\Sigma; \Z)$ is homogeneous of Alexander degree $j$ and Maslov degree $i+j$.
Ozsv\'ath and Szab\'o prove that $C(K_B)$ is quasi-isomorphic to $\CFKoo(K_B)$. 

The action of an element $\gamma\in H_1(\Sigma;\Z)$ on $C(K_B)$ is the following:
\[
\gamma \cdot (U^n\otimes \omega) = U^n\otimes \iota_\gamma\omega + U^{n+1}\otimes (\PD(\gamma)\wedge \omega).
\]
where $\iota_\gamma$ denotes contraction.

An easy check shows that the kernel of the action of $H_1(\Sigma;\Z)$ is generated over $\Z[U,U^{-1}]$ by the element $x=\prod(1+U\otimes a_i\wedge b_i)$, 
where $\{a_i,b_i\}$ is a symplectic basis of $H^1(\Sigma;\Z)$. Notice that each of the summands $x_I = \bigwedge_{i\in I\subset \{1,..,g\}} a_i\wedge b_i$ appears with nonzero coefficient in $x$, and that each chain $x_k = \sum_{\#I=k} x_I$ lives in the summand $C(K_B)\{0,-g+2k\}$.

We are now ready to prove Proposition~\ref{p:dbottom}.

\begin{proof}[Proof of Proposition~\ref{p:dbottom}]
Notice that $d^2$ is larger than $2\delta+2g = 2g(K_C)+2g$, since it follows from the degree-genus formula that $2g+2\delta = d^2-3d+2$ and the latter is smaller than $d^2$ for every $d\ge 1$. In particular, since 
$g(K) = g(K_C)+g$, $d^2>2g(K)$, equation \eqref{e:psurg} applies.

We are interested in the kernel of the $H_1$-action on the big complex $C^b_m(K)$  and the small complex $C^s_m (m(K))$.

By the K\"unneth formula for knot Floer homology~\cite[Theorem 7.1]{OSKI}, $C(K) = C(K_C)\otimes C(K_B)$.
Any element in $C(K_C)\otimes \langle x\rangle$ lies in the kernel of the action of $H_1(\Sigma;\Z)$, since $H_1(\Sigma;\Z)$ only acts on $C(K_B)$.
Moreover, since the differential on $C(K_B)$ is trivial, a chain $y\otimes z\in C(K)$ is a cycle (respectively, a boundary) if and only if $y$ is a cycle (resp., a boundary).

We start with proving the first equality.

Notice that $C(K_C)\otimes \langle x_k\rangle$ is a subcomplex of $C(K)$ which is isomorphic, as a graded module, to $C(K_C)$, with the degree shifted by $2k-g$; it follows that $\pi^b_m(C(K_C)\otimes \langle x_k\rangle)$ is isomorphic, as a graded complex, to $C^b_{m+g-2k}(K_C)$.

Hence, the minimal degree of elements in the image of $(\pi^b_m)_*:H_*(C(K_C)\otimes \langle x_k \rangle)\to H_*(\pi^b_m(C(K_C)\otimes \langle x_k\rangle))$ is $2k-g-2V_{m+g-2k}(K_C)$.

Recall now that $d_b(Y_C,\t_m)$ is defined as the minimal degree in the image of the kernel of the $H_1$-action under $(\pi^b_m)_*$, together with the degree shift given by~\eqref{e:psurg}.
Here the kernel is $H_*(C(K_C)\otimes \langle x\rangle)$, and the minimal degree in its image under $(\pi^b_m)_*$ is the minimal degree in the image of $H_*(C(K_C)\otimes \langle x_k \rangle)$ as $k$ varies from $0$ to $g$.

In particular, since the framing is $N=d^2$, it follows from \eqref{e:psurg} that 
\[
d_b(Y_C,\t_m) = \min_{0\le k \le g}\{2k-g-2V_{m+g-2k}(K_C)\} - \frac{d^2-(2m-d^2)^2}{4d^2},
\]
as desired.

We now turn to the second equality. Recall that we denoted with $2V'_m(m(K_C))$ the minimal degree of a nontorsion element in $H_*(C^s_m(m(K_C)))$, and that we computed in Lemma~\ref{lemma:Hm} that $V'_m(m(K_C)) = V_m(K_C)+m$. Moreover, since $d^2>2g(K)$, equation \eqref{e:nsurg} applies.

Since $K_B$ is amphicheiral, $m(K)$ is the connected sum of $m(K_C)$ and $K_B$, and the K\"unneth principle for knot Floer homology gives us a quasi-isomorphism $C(m(K)) \simeq C(m(K_C))\otimes C(K_B)$.
Similarly as for the big complex above, observe that $\pi^s_m(C(m(K_C))\otimes \langle x_k\rangle)$ is a subcomplex isomorphic to $C^s_{m+g-2k}(m(K_C))$, up to a degree shift by $g-2k$.
Thus, the minimal degree of a nontorsion element in $H_*(\pi^s_m(C(m(K_C))\otimes \langle x_k\rangle))$ is $2V'_{m+g-2k}(m(K_C))+g-2k$.
As above, the minimal degree in the image of the kernel of the $H_1$-action is now obtained by minimising this quantity when $k$ runs between 0 and $g$, and the second equality follows from~\eqref{e:nsurg}.
\end{proof}

\section{Proof of Theorem~\ref{thm:main}}\label{s:mainproof}

Recall from Equation~\eqref{e:RI} and from \cite[Propositions 4.4 and 4.6]{BL} the connection between the following functions associated to the link of a singularity: the gap function $I$, the semigroup-counting function $R$, and the function $V$ that computes the correction terms for positive surgeries. Here we suppress the singularity from the notation and we write $I(m), R(m), V(m)$ instead of $I_m, R_m, V_m$.

For every integer $m\ge \delta$, we have
\begin{equation}\label{e:VIR}
V({m-\delta}) = I(m) = R(m) + \delta - m
\end{equation}
where $2\delta$ is the Milnor number of the singularity.

Borodzik and Livingston explained in~\cite{BL} what happens when there are more singular points, i.e. when the knot $K_C$ is a \e{connected sum} of links of singularities.

\begin{defn}[{\cite[Equation (5.3)]{BL}}]
Given two functions $I, I': \Z\to\Z$ bounded from below, we denote with $I\diamond I'$ the \e{infimum convolution} of $I$ and $I'$:
\[
(I\diamond I')(s) = \min_{m\in\Z} I(m) + I'(s-m).
\]
\end{defn}

The property of the infimum convolution that is most relevant to us is the following refinement of the relation~\eqref{e:VIR} above.

\begin{lemma}[{\cite[Proposition 5.6]{BL}}]\label{l:morecusps}
Let $I^1,\dots, I^n$ be the gap functions associated to the links $K_1,\dots, K_n$ of $n$ cuspidal singularities. Let $K = \#_i K_i$ and $\delta = \sum \delta_i$, where $2\delta_i$ is the Milnor number of $K_i$. Finally, let $V$ be the function that computes the correction terms associated to $K$. Then
\[V(m) = (I^1\diamond \dots \diamond I^n)(m+\delta).\]
\end{lemma}

The following theorem is a generalisation of Theorem~\ref{thm:main} for arbitrary cuspidal curves; the statement reduces to Theorem~\ref{thm:main} by setting $n=1$.

\begin{thm}\label{t:morecusps}
Let $C$ be a genus-$g$ curve of degree $d$ with $n$ cusps, and let $I^1,\dots, I^n$ be the gap counting functions associated to its singularities. Then for every $-1\le j\le d-2$ and for every $0\le k \le g$ we have
\[
k-g \le (I^1\diamond\cdots\diamond I^n)(jd+1-2k)- \frac{(d-j-2)(d-j-1)}2  \le k.
\]
\end{thm}

\begin{proof}
Notice that both $W_C$ and $-W_C$ fulfil the hypotheses of Theorem~\ref{d-inequality}: in fact, Lemma~\ref{l:homology} asserts that $b^\pm_2(W_C) = 0$, so both $W_C$ and $-W_C$ are negative semidefinite; moreover, by Lemma~\ref{l:homology} and the universal coefficient theorem, $H^1(W_C;\Z)$ is trivial, so the restriction map $H^1(W_C;\Z)\to H^1(Y_C;\Z)$ is trivial; $\HFoo(-Y_C;\t)$ is standard in any torsion spin$^c$ structure $\t$, thanks to Lemma~\ref{l:standardHFoo}.

Notice that the image of the restriction map $\Spinc(W_C)\to\Spinc(Y_C)$ contains exactly $d$ torsion spin$^c$ structures. If $d$ is even, these are labelled by $hd$ where $h\in[-d/2,d/2]$ is a half-integer; if $d$ is odd, they are labelled by $hd$ where $h\in[-d/2,d/2]$ is an integer. Let us call $\s_h$ any spin$^c$ structure on $W_C$ that restricts to $\t_{hd}$ on $Y_C$.

Finally, let us observe that when $m=hd$ we have
\[
d_\pm(hd) = \frac{d^2-(2hd\pm d^2)^2}{4d^2} = \frac{1-(2h\pm d)^2}{4}.
\]

We now apply Theorem~\ref{d-inequality} to $(W_C,\s_h)$ and its boundary $(Y_C,\t_{hd})$; let $V(m) = V_m(K_C)$.

The left-hand side of the inequality~\eqref{eqn:d-ineq} vanishes, since $b_2^\pm(W_C) = 0$; also, $b_1(Y_C) = 2g$, therefore we get:
\begin{equation}
0\le g+\min_{0\le k \le g}\{2k-g-2V({hd+g-2k})\} - \frac{1-(d-2h)^2}{4}.
\end{equation}

Applying it to $-(W_C,\s_h)$ and its boundary $-(Y_C,\t_{hd})$ we get:
\begin{equation}
0\le  g + \min_{0\le k \le g}\{g-2k+2hd+2V({hd+g-2k})\} + \frac{1-(d+2h)^2}{4}.
\end{equation}

In particular, for every $h,k$ in the relevant ranges we have:
\[
k-g \le V({g-2k+hd}) - \frac{(d-2h)^2-1}{8} \le k.
\]

Rephrasing it in terms of the function $I = I^1\diamond\cdots\diamond I^n$, we get:
\[
k-g \le I((g+\delta)-2k+hd) - \frac{(d-2h)^2-1}{8} \le k.
\]

Since $2(g+\delta) = (d-1)(d-2)$, the substitution $j = h+(d-3)/2$ yields
\[
k-g \le I({jd+1-2k}) - \frac{(d-j-2)(d-j-1)}{2} \le k.\qedhere
\]
\end{proof}

\begin{rmk}\label{r:semigroup_property}
As many of the examples and applications become more transparent in the language of the semigroup counting function $R$ introduced earlier, we rephrase the inequalities in Theorem~\ref{thm:main} as follows:
\begin{equation}\label{eq:sem_ineq}
 0 \leq R({jd+1-2k}) + k - \frac{(j+1)(j+2)}{2} \leq g
\end{equation}
for every $j = -1, 0, \dots, d-2$ and $k = 0, \dots, g$.
\end{rmk}
\section{Unicuspidal curves with one Puiseux pair}\label{s:acc}

In this section we fix a positive integer $g$ and we restrict ourselves to unicuspidal curves 
of genus $g$ whose singularity has only one Puiseux pair $(a,b)$ with $a<b$; in this case, we say that the curve is $(a,b)$-\e{unicuspidal}. The degree-genus formula specialises to the following identity:
\begin{equation}\label{eq:dg_one_p}
 (d-1)(d-2) = (a-1)(b-1)+2g.
\end{equation}
This also means that, once we fix $g$, the pair $(a,b)$ determines uniquely the degree $d$ (unless $g=0$ and $a=1$, in which case we have no singularity).

\begin{defn} \label{def:admissible}
We say that a pair $(a,b)$ with $a < b$ is a \emph{candidate} (to be the Puiseux pair of a unicuspidal genus $g$ curve) if $a$ and $b$ are coprime and there is a positive integer $d$ such that the degree-genus formula~\eqref{eq:dg_one_p} holds.

If the corresponding semigroup counting function $R$ satisfies~\eqref{eq:sem_ineq} with the given genus $g$ for all possible values of $j$ and $k$, we say that the pair $(a,b)$ is an \emph{admissible candidate}.
\end{defn}

We are going to say that a certain property holds for \e{almost all} elements in a set if there are finitely many elements for which it does not hold.

Theorem~\ref{thm:pell} is a consequence of the following two propositions. Recall that we fixed the genus $g \geq 1$ of the curves we consider.

\begin{prop}\label{p:67}
If $g \geq 1$, then for almost all admissible candidates $(a,b)$ the ratio $b/a$ lies in the interval $(6,7)$.
\end{prop}

\begin{prop}\label{p:3d}
If $g \geq 1$, then for almost all admissible candidates $(a,b)$ such that $b/a\in(6,7)$ we have $a+b = 3d$.
\end{prop}

\begin{rmk}
We note here that in the proof of Proposition~\ref{p:67} we use recent work of Borodzik, Hedden and Livingston~\cite{BHL} to exclude the family $(a,b) = (l, 9l+1)$ in the case $g=1$.
\end{rmk}

\begin{proof}[Proof of Theorem~\ref{thm:pell}]
Combining the two propositions above, we get that almost all admissible pairs $(a,b)$ satisfy $a+b=3d$. If we substitute $3d=a+b$ in the degree-genus formula~\eqref{eq:dg_one_p} we readily obtain Equation~\eqref{eq:pell}.
\end{proof}

We prove Proposition~\ref{p:67} in Subsection~\ref{ss:67} and Proposition~\ref{p:3d} in Subsection~\ref{ss:3d}. We now turn to the proof of the corollaries stated in the introduction.

\begin{proof}[Proof of Corollary~\ref{c:finite}]
If $g\equiv2$ or $g\equiv 4$ modulo 5, the congruence $x^2\equiv 4(2g-1)\pmod 5$ has no solution, since $2$ and $3$ are not quadratic residues modulo 5. Hence Equation~\eqref{eq:pell} has no solution.
\end{proof}

We actually classify the genera $g$ such that Equation~\eqref{eq:pell} has a solution $(a,b)$ with $a,b$ coprime. This is done in Section~\ref{s:pell} below.

\begin{proof}[Proof of Corollary~\ref{c:acc}]
Plugging the relation $b = 3d - a$ into~\eqref{eq:dg_one_p} we get that for almost all pairs
\[ (d-1)(d-2) = (a-1)(3d - a - 1) + 2g. \]
Recall that $a$ is the multiplicity of the singularity, i.e.~the local intersection multiplicity of the singular branch with a generic line through the singular point. Therefore, due to B\'ezout's theorem, it can not be larger than $d$. So from the above equation one can compute
\[ a = \frac{3d-\sqrt{5d^2 + 4(2g-1)}}{2} \]
and notice that $\phi^2a - d$ is bounded for $a \geq 0, d \geq 0$.
\end{proof}

\subsection{The proof of Proposition~\ref{p:67}}\label{ss:67}

Before diving into the actual proof, we set up some notation and some preliminaries.

We are going to denote with $\N$ the set of non-negative integers, $\N = \{0,1,\dots\}$. Recall that the semigroup $\Gamma \subset \N$ associated to the singularity with Puiseux pair $(a,b)$ is generated by $a$ and $b$: $\Gamma = \langle a,b \rangle$. For any positive integer $n$ let us denote by $\Gamma(n)$ the $n$-th smallest element (with respect to the natural ordering of integers) of the semigroup $\Gamma$; for example, $\Gamma(1)$ is always 0 and $\Gamma(2)$ is always $a$.

We introduce the notation $\Delta_j$ for the triangular number $\frac{(j+1)(j+2)}{2}$. Setting $k = 0$ and using the lower bound in~\eqref{eq:sem_ineq}, for every $j = 0, 1, \dots, d-2$ we get the inequalities:
\begin{equation}\label{eq:spec1}
 \Delta_j \leq R_{jd+1}, \textrm{\ or, equivalently, \ } \Gamma(\Delta_j) \leq jd; \tag{$\star_j$}
\end{equation}
while setting $k = g$ and using the upper bound, for every $j = 0, 1, \dots, d-2$ we get:
\begin{equation}\label{eq:spec2}
 R_{jd+1-2g} \leq \Delta_j, \textrm{\ or, equivalently, \ } \Gamma(\Delta_j+1) > jd - 2g. \tag{$\star\star_j$}
\end{equation}

Every semigroup element can be expressed as $ub + va$ for some non-negative integers $u$ and $v$. Writing $\Gamma(\Delta_j) = ub+va$,~\hyperref[eq:spec1]{($\star_j$)} reads $ub/j + va/j \leq d$, and substituting this into the degree-genus formula~\eqref{eq:dg_one_p} we get:
\begin{equation}\label{eq:quad_1}
2g + (a-1)(b-1) \geq \left(\frac{u}jb + \frac{v}ja -1\right)\left(\frac{u}jb + \frac{v}ja-2\right).
\end{equation}

Analogously, if we write $\Gamma(\Delta_j+1) = ub+va$,~~\hyperref[eq:spec2]{($\star\star_j$)} reads $ub/j + va/j > d - 2g/j$, and substituting this into~\eqref{eq:dg_one_p} we get:
\begin{equation}\label{eq:quad_2}
2g + ab-a-b+1 < \left(\frac{u}jb + \frac{v}ja +\frac{2g-j}j\right)\left(\frac{u}jb + \frac{v}ja+\frac{2g-2j}j\right).
\end{equation}

Equations~\eqref{eq:quad_1} and~\eqref{eq:quad_2} give a prescribed region for admissible candidate pairs $(a,b)$; since the two inequalities are quadratic in $a$ and $b$, the boundary of such regions is a conic, typically a hyperbola. We are interested in the equation of the asymptotes of these hyperbolae, especially their slope. This motivates the following definition.

\begin{defn}\label{def:approx}
We say that for a set $\mathcal{P}\subset\N^2$ of pairs $(a,b)$ the \emph{asymptotic inequality} $\frac{b}{a} \preceq \alpha$ (respectively $\frac{b}{a} \succeq \alpha$) holds, if there is a constant $C$ such that $b \leq \alpha a + C$ (resp. $b \geq \alpha a + C$) for almost all pairs in $\mathcal{P}$. We say that $\frac{b}{a} \approx \alpha$ if both $\frac{b}{a} \preceq \alpha$ and $\frac{b}{a} \succeq \alpha$.
\end{defn}

\begin{rmk}
In this language, the Matsuoka--Sakai inequality means that $d/a \preceq 3$ and Orevkov's sharper result \cite{Or} means that $d/a \preceq \phi^2$. Notice that these results hold without any restrictions on the number of cusps and the Puiseux pairs: if we let $m_p$ be the minimal positive element in the semigroup of the singularity at $p$, then $a$ is replaced with $\max m_p$.
\end{rmk}

We now compute the slopes and equations of asymptotes of region boundaries arising from~\eqref{eq:quad_1} and~\eqref{eq:quad_2}.

\begin{lemma}\label{lemma:asympt}
Fix real constants $p,q,c_1,c_2,c_3$ such that $1-4pq>0$. If $p\neq 0$, let $\lambda_\pm = (1-2pq\pm\sqrt{1-4pq})/2p^2$. Let $D\subset \R^2$ be defined by the inequality
\[
(py + qx + c_1)(py + qx + c_2) \leq (y - 1)(x - 1) + c_3
\]
and let $((a_n, b_n))_{n\ge1}$ be a sequence of pairs of non-negative integers 
$a_n\leq b_n$ with $a_n\to\infty$.

If almost all pairs $(a_n, b_n)$ belong to $D$ then:
\begin{itemize}
\item if $p=0$, then $b_n/a_n \succeq q^2$;
\item if $p\neq 0$, then
\[\lambda_- \preceq \frac{b_n}{a_n} \preceq \lambda_+.\]
\end{itemize}

On the other hand, if almost all pairs $(a_n, b_n)$ do not belong to $D$, then:
\begin{itemize}
\item if $p=0$, then $b_n/a_n \preceq q^2$;
\item if $p\neq 0$, 
then either
\[\frac{b_n}{a_n} \preceq \lambda_- \textrm{ or } \lambda_+\preceq \frac{b_n}{a_n},\]
in the sense that the pairs can be divided into two subsets such that for the pairs in the first, resp. in the second subset the first, resp. the second asymptotic inequality holds.
\end{itemize}
\end{lemma}

\begin{proof}
If $p=0$, the conic has a vertical asymptote; the other asymptote is defined by the equation $(qx + c_1)(qx + c_2) = (y - 1)(x - 1) + c_3$ and has slope $q^2$, from which we immediately obtain that $b_n/a_n\succeq q^2$ if almost all pairs $(a_n,b_n)$ are in $D$, and $b_n/a_n\preceq q^2$ if almost all pairs are outside $D$.

A similar argument applies when $p\neq 0$. In this case, both asymptotes are non-vertical and their slopes are the solutions of the equation $p^2\lambda^2+(2pq-1)\lambda+q^2 = 0$, which are precisely $\lambda_\pm$. The analysis of the two cases is straightforward.
\end{proof}

\begin{rmk}\label{rem:const}
In some cases we will also need to compute (in terms of $p, q, c_1, c_2, c_3$) the largest constant $C_l$, respectively the smallest constant $C_s$, for which the following property holds: for every $\varepsilon > 0$, for almost all pairs $(a,b)$ satisfying the assumptions of the lemma, the suitable combination of inequalities (depending on the actual applicable statement of the lemma) of type
    \[ \lambda_{\pm}a + C_l - \varepsilon < b,\quad \textrm{resp.\ }  b < \lambda_{\pm}a + C_s + \varepsilon \]
holds. We will call such constants \emph{optimal}.
Rather than \e{a priori} computing the explicit constant, we will do it only when needed. Observe that the optimal constant is in fact the constant term in the normalized equation $y = Ax+C$ of the line which is the asymptote of the (relevant branch of the) hyperbola described by
  \[ (py+qx+c_1)(py+qx+c_2) = (y-1)(x-1)+c_3. \]
\end{rmk}

We set out to prove that $b/a$ lies in the interval $(6,7)$ for almost all admissible candidates $(a,b)$.

\begin{lemma}\label{lem:boundedpairs}
For every $M>0$ there are finitely many admissible candidates with $a<M$.
\end{lemma}

\begin{proof}
Suppose that there are infinitely many admissible candidates with $a<M$. Then for infinitely many candidates $b>3M>3a$ holds. From Equation~\eqref{eq:dg_one_p}  it follows that $d>3M+2g$ for infinitely many candidates; if $b > 3a$, however, the fourth semigroup element is $3a$. By~\hyperref[eq:spec2]{($\star\star_1$)}, we get $3a>d-2g>3M$, contradicting the assumption $a<M$.
\end{proof}

We handle the problem in seven cases, depending on the integer part of $b/a$. The general pattern of the proof in each case is the following. First, we choose an appropriate $j$ and we determine $u$ and $v$ such that $\Gamma(\Delta_j) = ub + va$ (respectively $\Gamma(\Delta_j + 1) = ub + va$). We then apply~\hyperref[eq:spec1]{($\star_j$)} or \hyperref[eq:spec2]{($\star\star_j$)} to get a quadratic inequality of type~\eqref{eq:quad_1} or \eqref{eq:quad_2}.

Of course, we can apply \hyperref[eq:spec1]{($\star_j$)} and \hyperref[eq:spec2]{($\star\star_j$)} only when $j \leq d-2$, but since for any bounded $d$ there are only finitely many admissible candidates by the degree-genus formula \eqref{eq:dg_one_p}, any result obtained by applying \hyperref[eq:spec1]{($\star_j$)} or \hyperref[eq:spec2]{($\star\star_j$)} with $j$ bounded will be valid for almost all admissible candidates. We wish to emphasize here that the actual upper bound for $j$ may (and in many cases indeed will) depend on the fixed genus $g$.

Finally, we apply Lemma~\ref{lemma:asympt} and Remark~\ref{rem:const} to obtain from \eqref{eq:quad_1} and \eqref{eq:quad_2} inequalities of the type
\[ \alpha a + C_1 - \epsilon < b < \beta a + C_2 + \epsilon, \]
valid for all but finitely many relevant admissible candidates for any choice of $\varepsilon > 0$.

In most of the cases, $\alpha$ and $\beta$ will be rational. If needed, we repeat the above process and choose new values of $j$ to get better estimates, until we get estimates with $\alpha = \beta$, i.e.~we arrive at a bound of type
\[ r a + C_1 \leq s b \leq r a + C_2 \]
with some constants $r, s, C_1, C_2$, where $r$ and $s$ are integers and $C_1$ and $C_2$ might depend on $g$. That is, we have an asymptotic equality $b/a \approx r/s$ rather than two asymptotic inequalities.

In this way, we reduce each case to a finite number of possible linear relations between $a$ and $b$, that is, relations of form $r a + C = s b$ with \e{integral} coefficients.

As soon as we have such a relation, we can ask the following question: is it possible for a pair $(a,b)$ satisfying this relation to be a candidate in the sense of Definiton~\ref{def:admissible}? Solving the degree-genus formula~\eqref{eq:dg_one_p} as a quadratic equation in $d$, we see that there is an integral solution for $d$ if and only if $4(a-1)(b-1) + 8g + 1 = K^2$, where we write $K$ instead of $(2d-3)$. Plugging in the linear relation between $a$ and $b$, we can show that the equation has very few solutions.

We are now ready to prove Proposition~\ref{p:67}. Recall that $g\geq 1$ is an arbitrary genus but it is fixed during the proof. Also, by Lemma~\ref{lem:boundedpairs}, for any fixed bound on $a$, there are at most finitely many admissible candidates $(a,b)$ (with degree $d$) for a genus $g$ 1-unicuspidal curve singularity; likewise, by the degree-genus formula \eqref{eq:dg_one_p}, for any fixed bound on the degree $d$, there are at most finitely many admissible candidates $(a,b)$.

\begin{proof}[Proof of Proposition~\ref{p:67}]
\textbf{Case I}: $1 < b/a \leq 2$. Choose $j = 1$. Now $b = \Gamma(\Delta_1) = \Gamma(3)$ and by~\hyperref[eq:spec1]{($\star_1$)} we have
$b\leq d$, and Lemma~\ref{lemma:asympt} with Remark~\ref{rem:const} implies $b < a + 1 + \epsilon$ for any $\epsilon > 0$ for almost all admissible candidates. So eventually $b = a+1$ for almost all admissible pairs (that is, $b/a \approx 1^2$). Plugging into the degree-genus formula \eqref{eq:dg_one_p} we obtain
\[
4(a-1)a + 8g + 1 = K^2 \Leftrightarrow (2a - 1)^2 + 8g = K^2
\]
and this is possible for infinitely many $a$ only if $g = 0$.

\vskip 0,2 cm
\textbf{Case II}: $2 < b/a \leq 4$. Choose $j = 1$, and observe that $2a = \Gamma(\Delta_1)$. By~\hyperref[eq:spec1]{$(\star_1)$},  $2a \leq d$, so by Lemma~\ref{lemma:asympt} and Remark~\ref{rem:const} $4a - 1 - \epsilon < b$, so $b = 4a - 1$ for almost all admissible pairs ($b/a \approx 2^2$).
\[
4(a-1)(4a-2) + 8g + 1 = K^2 \Leftrightarrow (4a - 3)^2 + 8g = K^2,
\]
and this equation has infinitely many solutions only if $g = 0$.

\vskip 0,2 cm
\textbf{Case III}: $4 < b/a \leq 5$. Choose $j = 2$ and apply~\hyperref[eq:spec1]{$(\star_2)$}: we obtain $b = \Gamma(\Delta_2) \leq 2d$, that is $b < 4a + 1 + \epsilon$, so $b = 4a + 1$ for almost all admissible pairs ($b/a \approx 2^2$).
\[
4(a-1)4a + 8g + 1 = K^2 \Leftrightarrow (4a - 2)^2 + 8g - 3 = K^2,
\]
and this is not possible for infinitely many $a$ for any non-negative $g$, as $8g - 3 \neq 0$.

\vskip 0,2 cm
\textbf{Case IV}: $5 < b/a \leq 6$. We have $5a=\Gamma(\Delta_2)$. Choose $j = 2$ and apply~\hyperref[eq:spec1]{$(\star_2)$}, so that $5a \leq 2d$, hence $b/a \succeq 25/4 > 6$, so there are at most finitely many admissible pairs $(a,b)$ in this case.

\vskip 0,2 cm
\textbf{Case V}: $7 < b/a \leq 8$. This needs a longer examination.

First choose $j = 3$, and apply~\hyperref[eq:spec1]{$(\star_3)$}: this yields $8a = \Gamma(\Delta_3) \leq 3d$, from which we obtain that for any $\varepsilon > 0$, for almost all admissible candidates belonging to this case,
\begin{equation}\label{eq:casev}
 64a/9 + 1/9 - \epsilon < b.
\end{equation}

Set $j = 4$ and notice that $11a = \Gamma(\Delta_4+1)$; using~\hyperref[eq:spec2]{$(\star\star_4)$} we obtain that $b/a \preceq (11/4)^2$.

We will now compare $2b$ with $15a$.

Assume first $2b > 15a$. Then from~\hyperref[eq:spec2]{$(\star\star_7)$} we get $2b+4a = \Gamma(\Delta_7 + 1) > 7d - 2g$, hence (by Lemma~\ref{lemma:asympt}) $\frac{33}{8} + \frac{7}{8}\sqrt{17} \preceq \frac{b}{a}$ (notice that the other asymptotic inequality is irrelevant in this case). Since $\frac{11^2}{4^2} < \frac{33}{8} + \frac{7}{8}\sqrt{17}$, combining the two asymptotic inequalities we just obtained, we proved that there are only finitely many admissible pairs $(a,b)$ in this subcase.

So we can assume $2b < 15a$. Again, from~\hyperref[eq:spec2]{$(\star\star_7)$} we get $19a = \Gamma(\Delta_7 + 1) > 7d - 2g$, hence $b/a \preceq (19/7)^2$.

\begin{claim}
$2b+7a = \Gamma(45)$.
\end{claim}
\begin{proof}
Notice that the elements preceding $2b+7a$ are exactly the following: $0, a, \dots, 21a; b, b+a, \dots, b+14a; 2b, 2b+a, \dots, 2b+6a$, as $2b+7a < 3b$ and $2b+7a < 22a$.
\end{proof}

From the claim above and~\hyperref[eq:spec1]{$(\star_8)$} we get $2b+7a = \Gamma(\Delta_8) \leq 8d$, which implies $b/a \preceq \frac{9}{2} + \frac{4}{\sqrt{2}}$.

In particular, we can assume $b/a < 22/3$, and hence $6b + 2a < 3b + 24a$. This means that the elements $44a$, $b+37a$, $2b+30a$, $3b+23a$, $4b+16a$, $5b+9a$, $6b+2a$ all precede $3b+24a$. So $ub+va < 3b+24a$ if $7u+v \leq 44$. There are $168$ semigroup elements of this form. In addition, the elements $45a < b + 38a < 2b + 31a$ also precede $3b+24a$. So $3b+24a$ is \emph{at least} the $172$nd semigroup element.
Using~\hyperref[eq:spec2]{$(\star\star_{17})$} we get
\[3b+24a \geq \Gamma(\Delta_{17} + 1) > 17d-2g \Longrightarrow b/a \preceq 64/9 = (8/3)^2.\]
Notice again that the other asymptotic inequality obtained by Lemma~\ref{lemma:asympt} is irrelevant since we are in the case $7 < b/a \leq 8$.

Coupled with~\eqref{eq:casev}, this means that $64a + 1 - \epsilon \leq 9b \leq 64a + C$ for some constant $C$. Thus, for any given positive integer $k$ and for any sufficiently large $a$,~\hyperref[eq:spec2]{$(\star\star_{6k+17})$} reads
\[3b + (16k + 24)a = \Gamma(\Delta_{6k+17} + 1) > (6k+17)d - 2g, \]
which shows that
\[9b \leq 64a + \frac{96g + 6k + 17}{6k+1} + \epsilon\]
holds for any $\varepsilon > 0$ for almost all admissible candidates belonging to this case.

The fraction on the right-hand side tends to 1 as $k \rightarrow \infty$, so we can chose a large enough $k$ (depending on the given fixed genus $g$) and a small enough $\varepsilon > 0$ such that $\frac{96g + 6k + 17}{6k+1} + \varepsilon < 2$. Combining this inequality  with~\eqref{eq:casev} applied with an $\varepsilon < 1/9$, we see that for all but finitely many admissible pairs $(a,b)$ belonging to this case the inequalities $64a < 9b < 64a + 2$ hold, that is, $64a + 1 = 9b$ for almost all admissible pairs in this case. Plugging this into the degree-genus formula~\eqref{eq:dg_one_p} we get:
\[4(a-1)\left(\frac{64a}9+\frac19-1\right) + 8g + 1 = K^2 \Leftrightarrow (16a - 9)^2 + 72g - 40 = 9K^2\]
and this has finitely many solutions $a$, as $72g - 40 \neq 0$.

\vskip 0,2 cm
\textbf{Case VI}: $8 < b/a \leq 9$. Since $b+3a = \Gamma(\Delta_4+1)$, from~\hyperref[eq:spec2]{$(\star\star_4)$} we get
$b + 3a > 4d - 2g$, and from Lemma~\ref{lemma:asympt} $9a - 6g - 2 - \epsilon < b \Rightarrow b/a \approx 9 = 3^2$.

So $9a - C < b < 9a$, and from this it is not hard to see that for any positive integer $k$ for every sufficiently large $a$ (and $b$) we have $\Gamma(\Delta_{6k+4} + 1) = (2k+1)b + 3a$. This means that using~\hyperref[eq:spec2]{$(\star\star_{6k+4})$} we get $(2k+1)b + 3a > (6k+4)d - 2g \Rightarrow 9a - \frac{6g+3k+2}{3k+1} - \epsilon \leq b$.

The lower bound tends to $9a - 1 - \epsilon$ as $k \rightarrow \infty$, so fixing a large enough $k$ depending on $g$ only and a small enough $\varepsilon > 0$, we obtain that $9a - 1 = b$ holds for almost all admissible pairs in this case.
\[4(a-1)(9a-2) + 8g + 1 = K^2 \Leftrightarrow (18a - 11)^2 + 72g - 40 = (3K)^2.\]

This is not possible for infinitely many $a$ for any non-negative $g$ as $72g - 40 \neq 0$.

\vskip 0,2 cm
\textbf{Case VII}: $9 < b/a$. Choose $j = 1$ and notice that $3a = \Gamma(\Delta_1+1)$; by~\hyperref[eq:spec2]{$(\star\star_1)$} we get
$3a > d - 2g$, hence, by Lemma~\ref{lemma:asympt}, $b/a \preceq 9 \Rightarrow b/a \approx 9 = 3^2$.

So we obtained $9a < b < 9a + C$, and from this it is not hard to see that for any positive integer $k$ for every sufficiently large $a$ (and $b$) we have $\Gamma(\Delta_{6k+4} + 1) = (18k + 12)a$, leading via~\hyperref[eq:spec2]{$(\star\star_{6k+4})$} to $(18k + 12)a > (6k+4)d - 2g \Rightarrow b \leq 9a + \frac{6g + 3k + 2}{3k+2} + \epsilon$.

The upper bound tends to $9a + 1 + \epsilon$ as $k \rightarrow \infty$, so (fixing again a large enough $k$ and a small enough $\varepsilon > 0$) we have $9a + 1 = b$ for almost all admissible pairs in this case.

\[
4(a-1)9a + 8g + 1 = K^2 \Leftrightarrow (6a - 3)^2 + 8g - 8 = K^2,
\]
which is possible for infinitely many $a$ only if $g = 1$.
This last family is excluded in~\cite{BHL}.
\vskip 0,2cm
This concludes the proof.
\end{proof}

\begin{rmk}\label{rmk:g01stprop}
Notice that we used the assumption $g \geq 1$ only in Cases I and II. If $g = 0$, from the proof above we get that almost all admissible candidates $(a,b)$ satisfying $b/a \notin (6,7)$ are either of the form $(a,b) = (l, l+1)$ for some $l \geq 2$ or of the form $(a,b) = (l, 4l-1)$ for some $l \geq 2$. These are the infinite families (a) and (b) of~\cite[Theorem 1.1]{BLMN1}.
\end{rmk}

\subsection{The proof of Proposition~\ref{p:3d}}\label{ss:3d}

Before turning to the proof of Proposition~\ref{p:3d}, we recall some basic facts about the Fibonacci numbers. The interested reader is referred to~\cite[Section 6]{Or} for further details.

Recall that we denote with $\phi$ the golden ratio, $\phi = \frac{1+\sqrt5}2$. The Fibonacci numbers are defined by recurrence as $F_0 = 0$, $F_1 = 1$, $F_{n+1} = F_n + F_{n-1}$; more explicitly, one can write
\[F_n = \frac{\phi^n - (-\phi)^{-n}}{\sqrt5}.\]

We collect in the following proposition some useful identities about the Fibonacci sequence, easily proved either by induction or by substituting the explicit formula above.

\begin{prop}\label{p:fibonacci}
The following identities hold for any integers $k\ge 2$ and any $l\ge 1$:
\begin{align}
& \gcd(F_{2l-1},F_{2l+1}) = \gcd(F_{2l-1},F_{2l+3}) = 1\label{eq:fib0}\\
& F_k^2 - F_{k-2}F_{k+2} = (-1)^k\label{eq:fib1}\\
& F_{k-2} + F_{k+2} = 3F_k\label{eq:fib2}\\
& F_{2l+3}^2F_{2l-1}^2 - F_{2l+1}^2(F_{2l+1}^2+2) = 1\label{eq:fib3}\\
& F_{2l+3}^2 + F_{2l+1}^2 - 3F_{2l+1}F_{2l+3} = -1\label{eq:fib4}\\
& \lim_{i\rightarrow\infty} F_{2i-1}^2\left(\phi^4 - \frac{F_{2i+1}^2}{F_{2i-1}^2}\right) = \frac{2}{5}\left(\phi^4-1\right)\label{eq:lim1}\\
&  \lim_{i\rightarrow\infty} F_{2i+1}^2\left(\frac{F_{2i-1}^2}{F_{2i+1}^2} - \phi^{-4}\right) = \frac{2}{5}\left(1-\phi^{-4}\right).\label{eq:lim2}
\end{align}
\end{prop}

In this subsection, $(a,b)$ will always denote a pair such that $6<b/a<7$. We want to prove that an admissible pair $(a, b)$ and the corresponding degree $d$ are tied by the relation $a+b = 3d$, with at most finitely many exceptions.

In the course of the proof, we will state several lemmas, systematically postponing their proof to Subsection~\ref{ss:technical}.

\begin{proof}[Proof of Proposition~\ref{p:3d}]
Using \hyperref[eq:spec1]{($\star_3$)} and the fact that $a+b=\Gamma(10)$, we get $a+b = \Gamma(\Delta_3) \leq 3d$, which in turn by Lemma~\ref{lemma:asympt} implies $\frac{b}{a} \preceq \phi^4$.

More precisely (see Remark~\ref{rem:const}), the relevant asymptote of the hyperbola $\gamma_0$ determined by $a+b=3d$ has equation $b = \phi^4a$.

In particular, for any $\epsilon>0$, for almost all admissible pairs we have
\begin{equation}\label{eq:caseviii}
b\le\phi^4a+\epsilon.
\end{equation}
One easily verifies that for every $D\ge 1$ the hyperbola $\gamma_D$ determined (via the degree-genus formula~\eqref{eq:dg_one_p}) by $a+b=3d-D$ lies below $\gamma_0$ in the relevant region $\{2 \leq a < b\}$.

As $a+b \leq 3d$, we only need show that only finitely many admissible pairs satisfy $a+b<3d$. It will turn out that the line $b = \phi^4 a$ plays a crucial role in the proof. We divide the region below it into infinitely many \emph{sectors} cut out by lines of slope $F_{2l+1}^2/F_{2l-1}^2$ ($l \geq 2$). Notice that the sequence $\left(F_{2l+1}^2/F_{2l-1}^2\right)_l$ is increasing in $l$ and tends to $\phi^4$ as $l\to\infty$.

\begin{defn}
The $l$\emph{-th sector} $S_l$ (for $l \geq 2$) is the open region in the positive quadrant bounded by lines $F_{2l+1}^2a = F_{2l-1}^2b$ and $F_{2l+3}^2a = F_{2l+1}^2b$. That is,
\[S_l = \left\{ (a,b) : F_{2l+1}^2a/F_{2l-1}^2 < b < F_{2l+3}^2a/F_{2l+1}^2 \right\}.\]
The $l$\emph{-th punctured sector} $S^*_l$ is defined as $S^*_l = S_l\setminus\{(F_{2l-1}, F_{2l+3})\}$.
\end{defn}

The following lemma takes care of the region below all sectors.

\begin{lemma}\label{lem:lowersec}
There are at most finitely many admissible candidates such that $b < 25a/4 = F_5^2a/F_2^2$.
\end{lemma}

\begin{figure}
\includegraphics[scale=0.3]{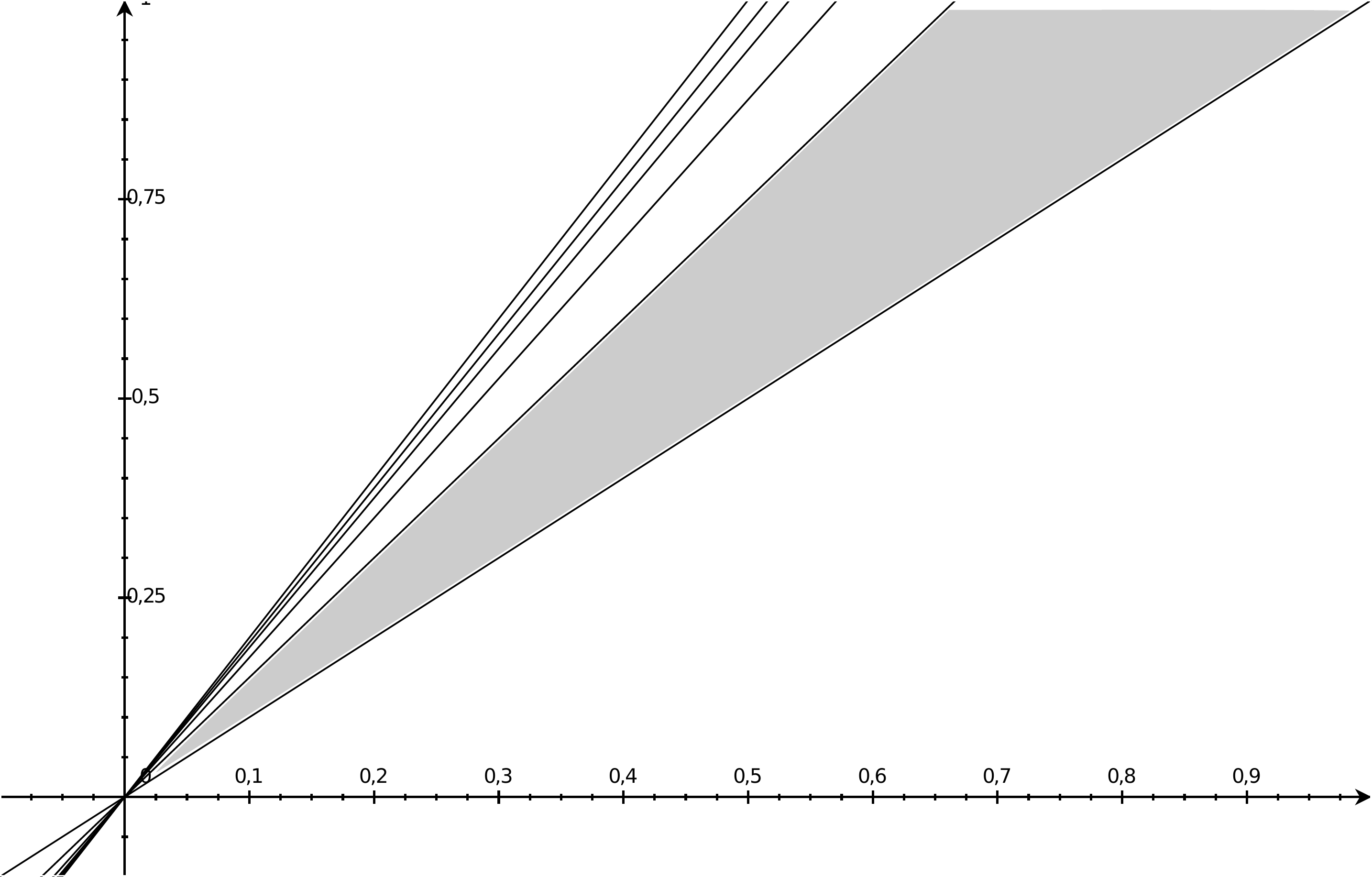}
\caption{A schematic picture of the sectors. The lowest sector is shaded.}
\end{figure}

The next lemma ensures that almost all admissible candidates below the line $b=\phi^4a$ live in the sectors $S_l$, $l\geq 2$ if $g \geq 1$. In fact, we prove more, namely that almost all admissible candidates below that line are in the \e{punctured} sectors.

\begin{lemma}\label{lem:exclg0}
For $g \geq 1$, there are only finitely many candidates of the form $(a, b) = (F_{2l-1}, F_{2l+3})$ and $(a, b) = (F_{2l-1}^2, F_{2l+1}^2)$, ($l \geq 2$).
\end{lemma}

For admissible candidates inside the sectors, we prove the following.

\begin{lemma}\label{lem:bound}
For any admissible candidate $(a, b) \in S_l$ at least one of the following upper bounds hold:
\begin{align}
a &\leq 2(2g-1)F_{2l+1}+2\label{eq:ub1}\\
b &\leq 2(2g-1)\frac{F_{2l+1}^2}{F_{2l-1}} + 2\label{eq:ub2}.
\end{align}
In particular, there are finitely many admissible candidates in each sector.
\end{lemma}

Recall that at the beginning of the proof we already obtained $a+b\leq 3d$ for almost all admissible candidates, so we want to prove that there are in fact only finitely many admissible candidates such that $a + b\le 3d-1$.
Notice that all pairs $(a, b)$ satisfying $a + b \leq 3d - 1$ and the degree-genus formula~\eqref{eq:dg_one_p} lie \emph{on or below} $\gamma_1$, which has an asymptote (in the relevant region $\{0 < a < b\}$) with equation $b = \phi^4a - 2\phi^2/\sqrt5$.

To finish the proof, we need one final lemma.

\begin{lemma}\label{lem:limzero}
There is a decreasing, infinitesimal sequence $(C_l)_{l\geq2}$ of real numbers (which further depends on $g$) such that for every $l\ge2$ and for every admissible candidate $(a,b)\in S_l$:
\[ 0 \le \phi^4a - b \le C_l. \]
\end{lemma}

Now we can show that almost all admissible candidates lie \emph{above} the line $b = \phi^4a - 1$. This is obviously true for pairs such that $b \geq \phi^4a$. To handle pairs below the line $b = \phi^4a$, first apply Lemmas ~\ref{lem:lowersec} and~\ref{lem:exclg0} and conclude that almost all admissible candidates in this case lie in the union of punctured sectors $S^*_l$ for $l \geq 2$. Now choose $l_0$ such that $C_{l_0} < 1 < (2\phi^2)/\sqrt5$.  From Lemma~\ref{lem:bound} above, we know that there are only finitely many admissible candidates in sectors $S_l$ with $l \leq l_0$.

So almost all admissible candidates are in sectors $S_l$ with $l > l_0$. For these, by Lemma~\ref{lem:limzero}, the inequality $\phi^4a - b < C_{l_0} < 1$ holds. Notice that for $g \geq 1$ the relevant branch (i.e.~the branch falling into the sector $\{0 < a < b\}$) of the hyperbola $\gamma_1$ lies \emph{above} its asymptote, and recall that the latter has equation $b = \phi^4a - 2\phi^2/\sqrt5$.

Denote by $(a_1, b_1)$ the intersection point of
$\gamma_1$ and the line $b = \phi^4a - 1$ in the positive quadrant, i.e.~$0 < a_1 < b_1$. One can easily compute that $a_1 = 2g$.


\begin{figure}
\includegraphics[scale=.3]{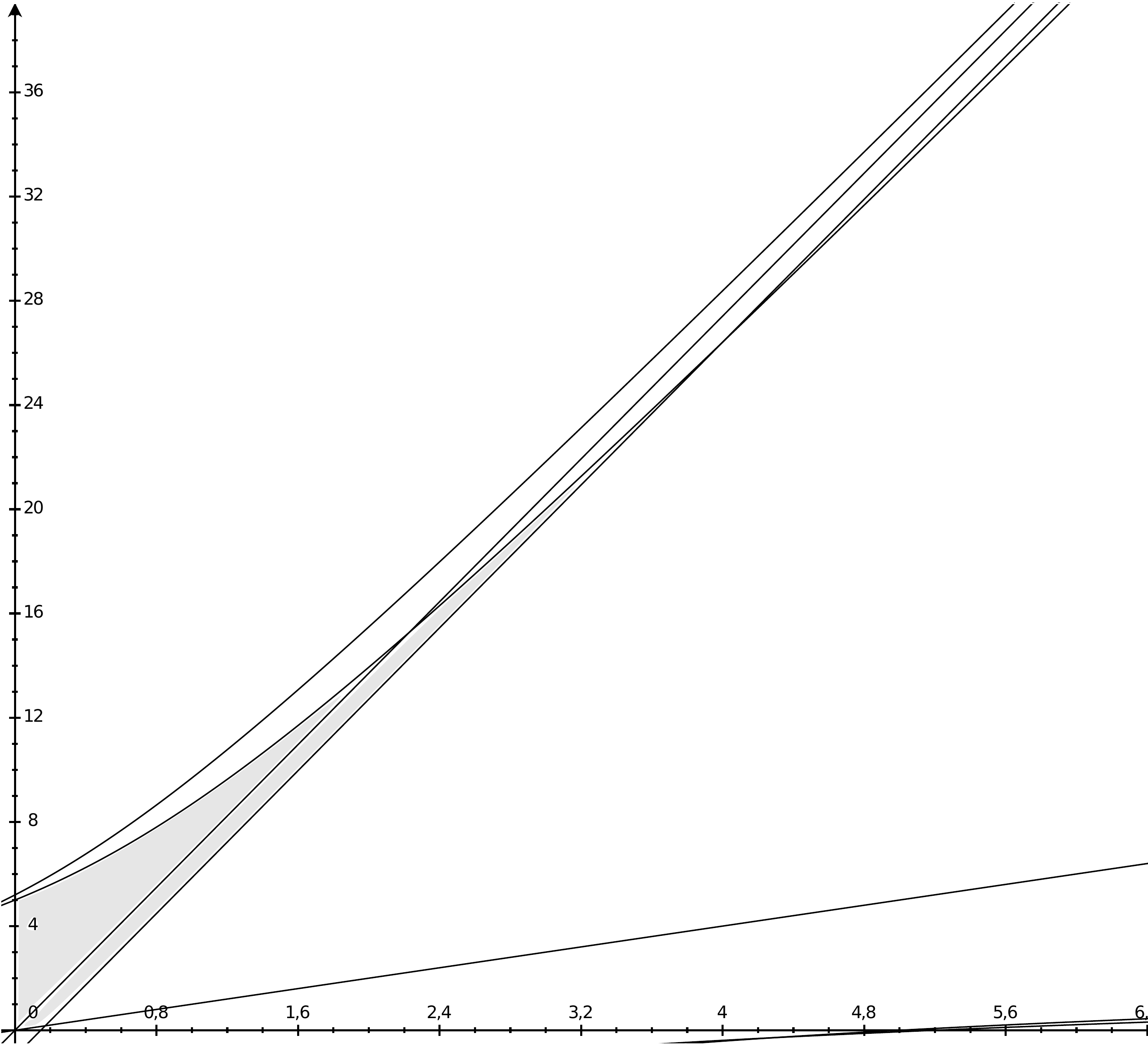}
\caption{The hyperbolae $\gamma_0$ and $\gamma_1$ in the $(a,b)$-plane, together with the lines $b=\phi^4a$, $b=\phi^4a-1$ and $b=a$}\label{f:gammas}
\end{figure}

We proved that almost all admissible candidates with $a + b < 3d$ lie \emph{on or below} $\gamma_1$ and \emph{above} the line $b = \phi^4a-1$ (by the argument involving Lemma~\ref{lem:limzero}, Lemma~\ref{lem:bound} and a suitable choice of $l_0$), and as these two intersect, almost all admissible candidates not satisfying~\eqref{eq:pell} lie in a bounded region $a \leq a_1 = 2g$ and $b \leq b_1$.

Therefore, almost all admissible candidates satisfy $a+b = 3d$.
\end{proof}

\begin{rmk}\label{rmk:g02ndprop}
By some small modifications of the argument, we are able to recover (up to finitely many candidates in a bounded region, which after working out a concrete bound, can be checked one by one by computer) the classification result of \cite[Theorem 1.1]{BLMN1} for $g=0$ as well. This is particularly interesting because our method \e{uses the semigroup distribution property of Remark~\ref{r:semigroup_property} only} (which, for $g=0$ is a result of \cite{BL}). After finishing this manuscript, we learned that Tiankai Liu in his PhD thesis \cite[Theorem 2.3]{Liu} among other results also reproved this classification based on the semigroup distribution property only.

So assume for the moment $g=0$. Recall that according to Remark~\ref{rmk:g01stprop} we can deal with candidates $(a,b)$ such that $b/a \notin (6,7)$. For the case $b/a \in (6,7)$, one can make the following changes to the proof above:
\begin{itemize}
\item The pairs listed in Lemma~\ref{lem:exclg0} are admissible candidates; in fact, 1-unicuspidal rational curves with those singularities exist: these are families (c) and (d) of~\cite[Theorem 1.1]{BLMN1}.

\item The asymptote of $\gamma_0$ still has equation $b = \phi^4a$, but $\gamma_0$ now lies \emph{below} it. So it is enough to deal with pairs below this line. This region is divided into sectors $S_l$. 

\item From the proof of Lemma~\ref{lem:bound}, we obtain that there are no admissible candidates in the \emph{punctured sectors} $S_l^{\ast}$. So, in fact, almost all admissible candidates are those already enumerated in Lemma~\ref{lem:exclg0}. 

\end{itemize}

In this way, we obtain an almost complete classification in the rational case: families (c) and (d) from~\cite[Theorem 1.1]{BLMN1} are obtained above, and families (a) and (b) were obtained in Remark~\ref{rmk:g01stprop}. We also see that almost all admissible candidates for $g=0$ belong to one of these families.
\end{rmk}

\subsection{Technical proofs}\label{ss:technical}
In this subsection we deal with all the lemmas stated above. The following claim will be useful in the proof of several lemmas.

\begin{claim}\label{c:smallestfrac}
Assume that we have three reduced fractions
\[ 0 < \frac{m_1}{n_1} < \frac{b}{a} < \frac{m_2}{n_2}\]
and set $P = m_2 n_1 - m_1 n_2$.
Then
\[ b \geq \frac{m_1 + m_2}{P} \textrm{\ and \ } a \geq \frac{n_1 + n_2}{P}. \]
\end{claim}

\begin{proof}
Write $b = \lambda_1m_1 + \lambda_2m_2$ and $a = \lambda_1n_1 + \lambda_2n_2$. Since $b/a$ falls between the two endpoints, $\lambda_1 > 0$ and $\lambda_2 > 0$. As $m_2 n_1 - m_1 n_2 = P$, due to Cramer's rule, the coefficients can be written in a form $\lambda_1 = \ell_1/P$ and $\lambda_2 = \ell_2/P$, for some integers $\ell_1$ and $\ell_2$, not necessarily coprime with $P$. Since they are also positive, $\ell_1 \geq 1$ and $\ell_2 \geq 1$, so $b \geq (m_1 + m_2)/P$ and $a \geq (n_1 + n_2)/P$.
\end{proof}

\begin{proof}[Proof of Lemma~\ref{lem:lowersec}]
Apply~\hyperref[eq:spec1]{$(\star_2)$}: $5a = \Gamma(\Delta_2) \leq 2d$ from which (via Lemma~\ref{lemma:asympt} and Remark~\ref{rem:const}) $25a/4 - 1/4 - \epsilon < b < 25a/4$. Therefore, for almost all admissible candidates $4b = 25a - 1$ holds. For the candidates lying on this line, however, $5a = 2d$ can not hold, as then $a$ would be even; a contradiction. So actually $5a \leq 2d - 1$, but the hyperbola determined by $5a \leq 2d - 1$ already intersects the line $4b = 25a - 1$, providing for $a$ the following upper bound: $a \leq 4g/5 + 1/5$.
\end{proof}

\begin{proof}[Proof of Lemma~\ref{lem:exclg0}]
For every $l\ge2$ the triples $(a,b,d_0) = (F_{2l-1}, F_{2l+3}, F_{2l+1})$ and $(a,b,d_0) = (F_{2l-1}^2, F_{2l+1}^2,F_{2l+1}F_{2l-1})$ satisfy the degree-genus formula~\eqref{eq:dg_one_p} with $g=0$: this is a consequence of Proposition~\ref{p:fibonacci} above. In particular, $4(a-1)(b-1) = (2d_0-3)^2 - 1$. (Notice that these triples are realised by families (c) and (d) of~\cite[Theorem 1.1]{BLMN1}.)

In general, solving~\eqref{eq:dg_one_p} as a quadratic equation for $d$, one sees that for a given $g$, $(a,b)$ is a
candidate if and only if $4(a-1)(b-1) + 8g + 1 = (2d-3)^2$. Comparing this with the above relation $4(a-1)(b-1) = (2d_0-3)^2 - 1$ we get that 
\[ (2d_0 - 3)^2 + 8g = (2d-3)^2 \]
which has only finitely many solutions for $d \neq d_0$ if $g \neq 0$.
\end{proof}

Before the proof of Lemma~\ref{lem:bound}, we need some preparation.

\begin{claim}\label{lem:secsmall}
Let $\displaystyle\frac{b}{a} \neq \frac{F_{2l+3}}{F_{2l-1}}$ be a reduced fraction in the open interval $\displaystyle\left(\frac{F_{2l+1}^2}{F_{2l-1}^2}, \frac{F_{2l+3}^2}{F_{2l+1}^2}\right)$ (where $l \geq 2$). Then $a > F_{2l-1}$ and $(a-1)(b-1) \geq F_{2l+1}(F_{2l+1}+1)$.
\end{claim}

\begin{proof}
Together with
\begin{equation}\label{eq:lowa}
a \geq F_{2l-1} + 1,
\end{equation}
we are going to prove that
\begin{equation}\label{eq:lowb}
b \geq \frac{F_{2l+1}^2+F_{2l+3}}{F_{2l-1}}.
\end{equation}
In fact, the above two inequalities imply that
\[ (b-1)(a-1) \geq F_{2l+1}^2 + F_{2l+3} - F_{2l-1} \geq F_{2l+1}^2 + F_{2l+1}, \]
where the last inequality follows from $F_{2l+3} = F_{2l+2} + F_{2l+1} > 2F_{2l+1}$.

We split the proof in four cases:
\begin{itemize}
\item[{\bf (i)}] $\displaystyle\frac{b}{a} \in \left(\frac{F_{2l+1}^2}{F_{2l-1}^2},\frac{F_{2l+3}}{F_{2l-1}}\right)$
\item[{\bf (ii)}] $\displaystyle \frac{b}{a} \in \left(\frac{F_{2l+3}}{F_{2l-1}} , \frac{F_{2l+1}^2 + 2}{F_{2l-1}^2}\right)$
\item[{\bf (iii)}] $\displaystyle \frac{b}{a} = \frac{F_{2l+1}^2 + 2}{F_{2l-1}^2}$
\item[{\bf (iv)}] $\displaystyle \frac{b}{a} \in \left(\frac{F_{2l+1}^2 + 2}{F_{2l-1}^2} , \frac{F_{2l+3}^2}{F_{2l+1}^2}\right)$
\end{itemize}

Notice that each fraction above is in reduced form (see \eqref{eq:fib0}, \eqref{eq:fib3}).

{\bf (i)} Using Equation~\eqref{eq:fib1}, we get that, in the notation of Claim~\ref{c:smallestfrac}, $P = F_{2l-1}^2F_{2l+3} - F_{2l+1}^2F_{2l-1} = F_{2l-1}$, so via Claim~\ref{c:smallestfrac} we immediately obtain~\eqref{eq:lowb} and~\eqref{eq:lowa}.

{\bf(ii)} If $b/a < (F_{2l+1}^2 + 2)/F_{2l-1}^2$, using Claim~\ref{c:smallestfrac}, we compute $P = F_{2l-1}$, hence~\eqref{eq:lowb} and~\eqref{eq:lowa} both hold (the estimate for $b$ is much larger than needed).

{\bf(iii)} When $b/a = (F_{2l+1}^2 + 2)/F_{2l-1}^2$,~\eqref{eq:lowb} reads:
\[ b = F_{2l+1}^2 + 2 \geq \frac{F_{2l+1}^2+F_{2l+3}}{F_{2l-1}}, \]
which follows from rearranging
\[ F_{2l+1}^2\left(F_{2l-1} - 1\right) \geq F_{2l}^2 \geq 3F_{2l} = F_{2l+3} - 2F_{2l-1}.\]
On the other hand, the inequality~\eqref{eq:lowa} is obvious.

{\bf(iv)} If $(F_{2l+1}^2 + 2)/F_{2l-1}^2 < b/a$, then using Claim~\ref{c:smallestfrac}, we get $P = 1$. This leads to
\[ b \geq F_{2l+3}^2 + F_{2l+1}^2 + 2 > \frac{F_{2l+1}^2+F_{2l+3}}{F_{2l-1}} \]
and
\[ a \geq F_{2l+1}^2 + F_{2l-1}^2 > F_{2l-1} + 1, \]
which show both~\eqref{eq:lowb} and~\eqref{eq:lowa}.
\end{proof}

In particular, the above claim says that for a pair $(a,b) \in S^*_l$ the assumptions of the next Lemma~\ref{lem:key} hold automatically.

\begin{lemma}\label{lem:key}
If for an admissible candidate $(a,b)$ we have $F_{2l+1} \leq d-2$, $F_{2l-1} < a$ and $(F_{2l+1}/F_{2l-1})^2 < b/a$, then one of the following two inequalities hold:
\[F_{2l+3}a \leq F_{2l+1}d\qquad {\rm or} \qquad F_{2l-1}b \leq F_{2l+1}d.\]
\end{lemma}

\begin{proof}
The key point is that due to the assumption $F_{2l+1} \leq d-2$ we can apply \hyperref[eq:spec1]{($\star_{F_{2l+1}})$}.

We count how many semigroup elements $ub + va$ can precede $F_{2l+3}a$. Since we assumed $(F_{2l+1}/F_{2l-1})^2 < b/a$, we can prove that $ub + va > F_{2l+3}a$ as soon as $uF^2_{2l+1} > F^2_{2l-1}(F_{2l+3}-v)$. So there is a chance to have $ub + va < F_{2l+3}a$ only if
\[
u F^2_{2l+1} \leq F^2_{2l-1}(F_{2l+3}-v).
\]

So bounding the number of semigroup elements that precede $F_{2l+3}a$ turns into the question of how many integer pairs $(u,v)$ satisfy $0 \leq u$, $0 \leq v$, and $u F^2_{2l+1} + v F^2_{2l-1} \leq F^2_{2l-1}F_{2l+3}$. Denote the set of these pairs by $H_l$ and its cardinality by $N_l$. Notice that the pair $(u,v) = (0, F_{2l+3})$ in $H_l$ corresponds to $F_{2l+3}a$. Later, it will be important that $(u,v) = (F_{2l-1}, 0) \in H_l$ as well (see~\eqref{eq:fib1}), i.e.~the corresponding element, $F_{2l-1}b$ can precede $F_{2l+3}a$.

\begin{claim}
The cardinality of $H_l$ is $N_l = \Delta_{F_{2l+1}} + 1$.
\end{claim}
\begin{proof}
Notice that $N_l$ is the number of integral lattice points \emph{on the boundary or in the interior} of the triangle $T$ with vertices given by coordinates $O = (0,0)$, $A = (0, F_{2l+3})$ and $C = (F_{2l+3}\frac{F^2_{2l-1}}{F^2_{2l+1}}, 0) = (F_{2l-1} + \frac{F_{2l-1}}{F^2_{2l+1}}, 0)$ (use~\eqref{eq:fib1}).

We will count the integral lattice points in the interior or on the boundary of a smaller triangle $T'$ with integral lattice point vertices given by coordinates $O = (0,0)$, $A = (0, F_{2l+3})$ and $B = (F_{2l-1}, 0)$ instead. This number will be $N_l$ as well, as there is no lattice point in the closure of the difference $T \setminus T'$ (triangle $ABC$) except points $A$ and $B$.
To see this, assume that there is such a lattice point $P$ with coordinates $(u,v)$ in the triangle $ABC$. Set $0 \leq s := F_{2l+3} - v \leq F_{2l+3}$ and $r := u$, and compare the slopes of lines $AC$, $AB$ and $AP$: the existence of the point $P$ would mean that the slope of $AP$ (which is $-s/r$) is either strictly between the slopes of $AB$ and $AC$ (being $-F_{2l+3}/F_{2l-1}$ and $-F_{2l+1}^2/F_{2l-1}^2$, respectively), or coincides with one of them. But this is a contradiction, since the fractions $F_{2l-1}/F_{2l+3}$ and $F_{2l-1}^2/F_{2l+1}^2$ are reduced (see \eqref{eq:fib0}), so $P$ can not be on $AB$ or $AC$; and there is no rational number $r/s$ such that $s \leq F_{2l+3}$ and
\[
 \frac{F_{2l-1}}{F_{2l+3}} < \frac{r}{s} < \frac{F^2_{2l-1}}{F^2_{2l+1}}.
\]
To prove the above fact, use Claim~\ref{c:smallestfrac} and notice that $P = F_{2l+3}F_{2l-1}^2 - F_{2l-1}F_{2l+1}^2 = F_{2l-1}$ (use~\eqref{eq:fib1}), and get $s \geq (F_{2l+3} + F_{2l+1}^2)/F_{2l-1} > F_{2l+3}$. (For this last inequality use again~\eqref{eq:fib1} and the trivial fact that $F_{2l+3} > 1$.)

Since $F_{2l-1}$ and $F_{2l+3}$ are coprime, there are no lattice points on the hypotenuse of the triangle $T'$ other than the endpoints. In this way, $N_l = 1 + \frac{1}{2} (F_{2l-1} + 1)(F_{2l+3} + 1)$ (half of the number of the lattice points in the appropriate closed rectangle, plus one endpoint of the hypotenuse), which, using~\eqref{eq:fib2} further equals $1 + \frac{1}{2} (F_{2l+1} + 1)(F_{2l+1} + 2) = 1 + \Delta_{F_{2l+1}}$.
\end{proof}

This means that at most $\Delta_{F_{2l+1}}$ semigroup elements can precede $F_{2l+3}a$ (remember that $(0, F_{2l+3}) \in H_l$). So $F_{2l+3}a$ is \emph{at most} the $(\Delta_{F_{2l+1}}+1)$-th element: $F_{2l+3}a \leq \Gamma(\Delta_{F_{2l+1}}+1)$.

If $F_{2l+3}a$ was not the $(\Delta_{F_{2l+1}}+1)$-th, then by~\hyperref[eq:spec1]{($\star_{F_{2l+1}}$)}, we would have
\[ F_{2l+3}a \leq \Gamma(\Delta_{F_{2l+1}}) \leq F_{2l+1}d \]
that is the first inequality we were looking for.

On the other hand, if $F_{2l+3}a$ is the $(\Delta_{F_{2l+1}}+1)$th element, i.e.~$F_{2l+3}a = \Gamma(\Delta_{F_{2l+1}}+1)$, then \emph{all the semigroup elements corresponding to integer pairs in} $H_l$ \emph{have to be smaller than} $F_{2l+3}a$ (except of course $F_{2l+3}a$ itself). In particular, $F_{2l-1}b < F_{2l+3}a$ (equality here can not hold for $a > F_{2l-1}$ due to coprimality), so $F_{2l-1}b$ is \emph{at most} the $\Delta_{F_{2l+1}}$-th element. In this case, applying~\hyperref[eq:spec1]{($\star_{F_{2l+1}})$} we have \[F_{2l-1}b \leq \Gamma(\Delta_{F_{2l+1}})\leq F_{2l+1}d. \]

Thus the proof of Lemma~\ref{lem:key} is completed.
\end{proof}

\begin{proof}[Proof of Lemma~\ref{lem:bound}]

First notice that both estimates are true for $(a,b) = (F_{2l-1}, F_{2l+3})$, therefore we can assume that $(a, b) \in S^*_l$, thus we can apply Claim~\ref{lem:secsmall} and Lemma~\ref{lem:key}.

Assume that for an admissible candidate $(a,b)$ in the $l$-th sector the first inequality of Lemma~\ref{lem:key} holds: that is, the quantity $r = F_{2l+1}d - F_{2l+3}a$ is non-negative. Let $s = F_{2l+3}^2a - F_{2l+1}^2b \geq 1$.
A direct computation of the intersection of the line $F_{2l+3}^2a - F_{2l+1}^2b = s = $ constant and the hyperbola $F_{2l+1}d - F_{2+3}a = r = $ constant yields
\begin{equation}\label{eq:aintersec}
 a = (2g-1)\frac{F_{2l+1}^2}{s+2F_{2l+3}r-1} + \frac{s + 3F_{2l+1}r - r^2}{s + 2F_{2l+3}r - 1}.
\end{equation}
(To obtain this, from the two equations defining $s$ and $r$ express $b$ and $d$ in terms of $a,r,s$, then substitute into the degree-genus formula \eqref{eq:dg_one_p}; express $a$ in terms of $r,s$ and finally use the identity \eqref{eq:fib4}.)

If $r = 0$, then, as $F_{2l+1}$ and $F_{2l+3}$ are coprime, $F_{2l+1}$ divides $a$, so $s$ is divisible by $F_{2l+1}$ as well. In particular, $s \geq F_{2l+1}$. So we can estimate the right hand side of the above expression as follows:
\[ a \leq (2g-1)\frac{F_{2l+1}^2}{F_{2l+1}-1} + 2. \]

If $r \geq 1$, then (using $s \geq 1$ as well) we get the following upper bound:
\[ a \leq (2g-1)\frac{F_{2l+1}^2}{2F_{2l+3}} + 1. \]

The upper bound given in~\eqref{eq:ub1} is a generous upper estimate for both of the above bounds.

If the second case of Lemma~\ref{lem:key} holds, introduce notations $s = F_{2l-1}^2b - F_{2l+1}^2a$ and $r = F_{2l+1}d - F_{2l-1}b$. In a similar way as above, we see that~\eqref{eq:ub2} is a (rather generous) upper bound for $b$.
\end{proof}

Observe that in the case $g=0$ the statement of Lemma~\ref{lem:bound} does not hold; however, a similar computation shows that for any admissible candidate in the \emph{punctured} sector $(a,b)\in S^{\ast}_l$ either $a < 2$ or $b < 2$ holds (or both): in fact, the first summand of the expression \eqref{eq:aintersec} is negative in this case and the second is at most $2$. So the above proof also shows that \emph{there are no admissible candidates in} $S_l^{\ast}$ for $g=0$. See also Remark~\ref{rmk:g02ndprop}.

\begin{proof}[Proof of Lemma~\ref{lem:limzero}]
It is obvious that $0 \leq \phi^4a-b$, as the pair $(a,b)$ is assumed to be in the sector $S_l$ which is below the line $b = \phi^4a$. To obtain the upper bound, we use Lemma~\ref{lem:bound}. When~\eqref{eq:ub1} holds, using~\eqref{eq:lim1} we can write:
\begin{align*}0 &\leq \phi^4a - b \leq  \phi^4a - \frac{F_{2l+1}^2}{F_{2l-1}^2}a = a\left(\phi^4-\frac{F_{2l+1}^2}{F_{2l-1}^2}\right) \\
&\leq \left(2(2g-1)\frac{1}{F_{2l+1}}F_{2l+1}^2+2\right)\left(\phi^4-\frac{F_{2l+1}^2}{F_{2l-1}^2}\right) \rightarrow 0.
\end{align*}

On the other hand, when~\eqref{eq:ub2} holds, using~\eqref{eq:lim2} we can write:
\begin{align*}0 &\leq \phi^4a - b \leq  \phi^4\frac{F_{2l-1}^2}{F_{2l+1}^2}b - b = \phi^4b\left(\frac{F_{2l-1}^2}{F_{2l+1}^2} - \phi^{-4}\right)\\
&\leq \phi^4 \left(2(2g-1) \frac{1}{F_{2l-1}} F_{2l+1}^2 + 2 \right)\left(\frac{F_{2l-1}^2}{F_{2l+1}^2} - \frac1{\phi^4}\right) \rightarrow 0.\qedhere
\end{align*}
\end{proof} 
\section{A generalised Pell equation}\label{s:pell}

This section is a short trip in number theory, in which we study the solutions of the generalised Pell equation
\begin{equation}\label{e:gpell}
x^2-5y^2 = n \tag{$\varspadesuit_n$}
\end{equation}
as $n$ varies among the integers. This is closely related to Equation~\eqref{eq:pell}. In particular, we will be determining the values of $n$ for which there exists a solution $(x,y)$ to~\eqref{e:gpell} where $x$ and $y$ are coprime; for these $n$ there are infinitely many such pairs, and this allows us to generalise Corollary~\ref{c:finite} and prove Theorem~\ref{t:construction}. Along the way, we will also introduce some notation that we will use in the next section.

We will work in the ring $\O = \O_K$ of integers of the real quadratic field $K=\Q(\sqrt5)$; there is an automorphism on $\O$, that we call \e{conjugation} and denote with $\alpha\mapsto\overline\alpha$, that is obtained by restricting the automorphism of $K$ (as a $\Q$-algebra) that maps $\sqrt5$ to $-\sqrt5$.

Given $\alpha = x+y\sqrt5\in\O$ we will call $N(\alpha) = \alpha\cdot\overline\alpha = x^2 - 5y^2$ the \emph{norm} of $\alpha$; notice that if $N(\alpha)$ is prime, then $\alpha$ itself is prime. The element $\phi = \frac{1+\sqrt5}2 \in \O$ has norm $-1$, hence it is a unit.

We begin by collecting some well-known facts about $\O$. (See, for example,~\cite{Marcus}, Chapters 2 and 3.)

\begin{thm}[{\cite[Chapter 2]{Marcus}}]
The ring $\O$ has the following properties:
\begin{itemize}
\item $\O$ is generated (as a ring) by $\phi$, {i.e.}~$\O = \Z[\phi]$;
\item $\O$ is a Euclidean ring, hence it is a principal ideal domain (PID);
\item the group of units $\O^*$ of $\O$ is isomorphic to $\Z\oplus\Z/2\Z$, and the isomorphism maps $\phi$ to $(1,0)$ and $-1$ to $(0,1)$; in particular, elements of norm $1$ are of the form $\pm\phi^{2h}$ for some integer $h$.
\end{itemize}
\end{thm}

Since $\O$ contains $\Z[\sqrt5]$, whose additive group is isomorphic to $\Z^2$, in what follows we will frequently identify a pair of integers $(x,y)$ with the algebraic integer $x+y\sqrt5$. In particular, we will identify a solution $(x,y)$ of Equation~\eqref{e:gpell} with the associated element $x+y\sqrt5$ of $\O$ of norm $n$.

\subsection{Solving equation~\eqref{e:gpell}}
We start by looking at equation ~\hyperref[e:gpell]{$(\varspadesuit_p)$} when $p$ is a prime. We recall a standard result about factorisation of primes in $\mathcal{O}$.

\begin{thm}[{\cite[Theorem 25]{Marcus}}]
Given a prime $p\in\Z$, consider the ideal $P = p\O$:
\begin{itemize}
\item[(i)] $P$ is prime if and only if $p\equiv\pm2\pmod 5$;
\item[(ii)] $P=Q^2$ for some prime ideal $Q\subset\O$ if and only if $p=5$;
\item[(iii)] $P=QQ'$ for two distinct prime ideals $Q, Q'\subset\O$ if and only if $p\equiv\pm1\pmod5$; moreover, in this case $Q' = \overline Q$.
\end{itemize}
Additionally: if $(5) = Q_1Q_2$ for some prime ideals $Q_1,Q_2$, then $Q_1 = Q_2 = (\sqrt5)$; if $p\equiv\pm1\pmod 5$ and $Q_1\overline{Q_1}=Q_2\overline{Q_2}$ are two prime factorisations of $p\O$, then either $Q_1 = Q_2$ or $Q_1 = \overline{Q_2}$.
\end{thm}

In particular, since $\O$ is a PID, the ideal $Q$ of point (iii) above is generated by an element $\alpha\in\O$, whose norm $N(\alpha)$ is $\pm p$. Up to multiplying with $\phi$, we can assume that $N(\alpha) = p$.

\begin{claim}\label{c:integral}
The ideal $Q$ is generated by an element of norm $p$ in $\Z[\sqrt5]\subset\O$.
\end{claim}

\begin{proof}
Let $\alpha$ be a generator $\alpha$ of $Q$ with norm $p$. Write $\alpha = u+v\phi$. If $v$ is even, there is nothing to prove; therefore, let us suppose that $v$ is odd.

If $u$ is even, consider $\alpha' = \phi^2 \alpha$; $N(\alpha') = N(\alpha)N(\phi)^2 = p$, and, since $\phi^2 = \phi+1$, we have
\[
\alpha' = \phi^2\cdot(u+v\phi) = u\phi + u + 2\phi v + v = u+v + (u+2v)\phi,
\]
and $u+2v$ is even.

Analogously, if $u$ is odd, consider $\alpha' = \overline\phi{\vphantom\phi}^2 \alpha$; as above, $N(\alpha') = p$, and, since $\overline\phi = 1-\phi$ and $\overline\phi{\vphantom\phi}^2 = 2-\phi$, we have
\[
\alpha' = \overline\phi{\vphantom\phi}^2(u+v\phi) = 2u - u\phi - v(1-\phi) = 2u-v + (v-u)\phi,
\]
and $v-u$ is even.

In either case, $\alpha'\in\Z[\sqrt5]$ and has norm $p$.
\end{proof}

In particular, for every prime $p\equiv0,\pm1\pmod 5$ we have found two integers $x,y\in \Z$ such that $x^2-5y^2=p$. We let $\alpha_p = x+y\sqrt5$ and we call it a \e{fundamental solution} of the generalised Pell equation~\hyperref[e:gpell]{$(\varspadesuit_p)$}. If $p\neq 5$, we also have that $\alpha_p$ and $\overline\alpha_p$ are coprime.

We now turn to the existence of solutions to the generalised Pell equation~\eqref{e:gpell}. We will need the following lemma.

\begin{lemma}\label{l:qtoq2}
Let $q$ be an odd prime, $q\equiv\pm2\pmod 5$. If $q$ divides $n$ and $x^2-5y^2 = n$, then $q$ divides both $x$ and $y$.
\end{lemma}

\begin{proof}
Observe that, since $q$ divides $n$, $q\mid x$ if and only if $q\mid y$. The equation $x^2-5y^2 \equiv 0\pmod{q}$ has a nonzero solution if and only if 5 is a quadratic residue. Applying quadratic reciprocity we get:
\[\left(\frac{5}q\right) = (-1)^{\frac{(5-1)(q-1)}4}\left(\frac{q}5\right) = \left(\frac{q}5\right)\]
hence we have a solution if and only if $q\equiv0,\pm1 \pmod 5$. On the other hand, if $q\equiv\pm2$, the only solution is trivial, therefore $q$ divides $x$ and $y$.
\end{proof}

\begin{prop}\label{p:pell-solutions}
Equation~\eqref{e:gpell} has a solution if and only if all prime factors of $n$ that are congruent to $\pm 2$ modulo $5$ appear with an even exponent.
\end{prop}

\begin{proof}
We prove that the condition is necessary, first. Suppose that there is an integer $n$ and a prime $p\equiv\pm2\pmod5$ such that the $p$-adic valuation $v_p(n)$ is odd and~\eqref{e:gpell} has a solution $(x,y)$. Up to multiplying $x+y\sqrt5$ by $2+\sqrt5$, we can suppose that $n$ is positive, and that it is minimal among all positive integers having this property.

If $p=2$, then either $x$ and $y$ are both odd, or they are both even. In the first case, $n = x^2-5y^2 \equiv 1-5\equiv 4\pmod 8$, hence $v_2(n) = 2$, contradicting the assumption that $v_2(n)$ be odd; in the second case, both $x$ and $y$ are divisible by $2$, hence $n/4 = (x/2)^2-5(y/2)^2$ and $v_2(n/4) = v_2(n)-2$ is odd, contradicting the minimality of $n$.

If $p>2$ is an odd prime, Lemma~\ref{l:qtoq2} implies that $p$ divides both $x$ and $y$, hence $(x/p, y/p)$ is a solution of $(x/p)^2-5(y/p)^2 = n/p^2$, and $v_p(n/p^2) = v_p(n)-2$ is still odd, thus contradicting the minimality of $n$.

We now prove that the condition is sufficient. If $(x_0,y_0)$ is a solution of Equation~\hyperref[e:gpell]{$(\varspadesuit_n)$}, then $(kx_0, ky_0)$ is a solution of Equation~\hyperref[e:gpell]{$(\varspadesuit_{k^2n})$}. Therefore, it is enough to find a solution of Equation~\eqref{e:gpell} when $n$ is squarefree. In particular, $n$ is a product of distinct primes that are congruent to $0$ or $\pm1$ modulo 5.

For every such prime $p$ we have produced an algebraic integer $\alpha_p\in\Z[\sqrt5]$ such that $\alpha_p\cdot\overline\alpha_p = p$ (see Claim~\ref{c:integral}). The product $\alpha_n = \prod_{p\mid n}\alpha_p$, is an integer $\alpha_n\in\Z[\sqrt5]$ such that $\alpha_n\cdot\overline\alpha_n = n$. Writing $\alpha_n = x+y\sqrt5$, we have a solution $(x,y)$ of~\eqref{e:gpell}.
\end{proof}

Finally, we refine the last proposition to get coprime solutions. We will call $(x,y)$ a \e{coprime solution} of~\eqref{e:gpell} if $\gcd(x,y) = 1$.

\begin{prop}\label{p:coprime}
Equation~\eqref{e:gpell} has a coprime solution $(x,y)$ if and only if $n=an'$ where $a\in\{1,4,5,20\}$ and $n'$ has no prime factors congruent to $0$ or $\pm 2$ modulo $5$.
\end{prop}

\begin{proof}
We first prove that, if $8$ divides $n$,~\eqref{e:gpell} has no coprime solutions. In fact, if there was such a solution, $x$ and $y$ would both be odd. Then
\[x^2-5y^2 \equiv 1-5\cdot1 = -4 \not\equiv 0 \pmod 8.\]
That is to say, if~\eqref{e:gpell} has a coprime solution, either $n$ is odd or $n=4m$ for some odd integer $m$, thanks to Proposition~\ref{p:pell-solutions}.

We now prove that, if $n=25n_1$, there are no coprime solutions. In fact, if 25 divides $n$, then $5$ divides $x$, hence $x=5x_1$. Dividing by $5$, we get the equation $5x_1^2-y^2 = 5n_1$, from which 5 divides $y$. This means that if~\eqref{e:gpell} has a coprime solution, 25 does not divide $n$.

On the other hand, there are coprime solutions when $n=4$ and $n=5$, namely $(3,1)$ and $(5,2)$.

Lemma~\ref{l:qtoq2} rules out all odd primes in the factorisation of $n$ that are congruent to $\pm2$ modulo 5.

We now prove that if $p\equiv\pm1\pmod5$ is a prime and $n=p^k$ is a power of $p$, then~\eqref{e:gpell} has a coprime solution: consider a fundamental solution $(x_1,y_1)$ for $p$, corresponding to the integer $\alpha_p = x_1 + y_1\sqrt5\in\O$ for $p$. Consider the solution $(x_k,y_k)$ corresponding to $\alpha_p^k$: $\gcd(x_k,y_k)$ divides $\alpha_p^k\overline\alpha_p^k = p^k$, hence is a power of $p$.
If $p = \alpha_p\overline\alpha_p$ divided $\gcd(x_k,y_k)$, though, then $\overline\alpha_p$ would divide $x_k$ and $y_k$, hence it would also divide $\alpha_p^k$, thus contradicting the fact that $\O$ is a UFD and that $\alpha_p$ and $\overline\alpha_p$ are coprime.

The same kind of argument shows that multiplying together fundamental solutions associated to each of the powers of primes
\[\{4,5\}\cup\{p^k\mid p\equiv\pm1\pmod5 , k\ge1\}\]
we get coprime solutions of the original equation.
\end{proof}

\begin{rmk}\label{r:omega}
We can actually say more, by looking at the proof of Proposition~\ref{p:coprime}. In what follows, we let $n = an'$ as above, and we denote with $\omega(n')$ the number of distinct prime factors of $n'$.

Since $\O$ is a PID, \emph{every} solution of~\eqref{e:gpell} is obtained by multiplying together fundamental solutions associated to the prime divisors of $n'$ and a solution of~\hyperref[e:gpell]{$(\varspadesuit_{a})$}. Now suppose that $p^2$ divides $n$ for some prime $p$, and that we choose to use \emph{both} fundamental solutions $\alpha_p$ and $\overline\alpha_p$ to produce a solution of~\eqref{e:gpell}. Recall that, by assumption, $p$ is an odd prime congruent to $\pm1$ modulo 5.

This means that we are considering the number $\alpha_n = \alpha_p\cdot\overline\alpha_p\cdot\beta$ for some $\beta\in\Z[\sqrt5]$.
But $\alpha_p\cdot\overline\alpha_p= p$, hence $\alpha_n = p\cdot\beta$, and if we write $\alpha_n = x+y\sqrt5$, then $x$ and $y$ are both divisible by $p$, and the solution we obtain is \e{not} coprime.

Therefore, if we want to obtain a coprime solution, the only choice we have is to use either $\alpha_p^v$ or $\overline\alpha_p^v$, where $v$ is the exponent of $p$ in the factorisation of $n$; moreover, since $p\neq5$, $\alpha_p$ and $\overline\alpha_p$ are coprime, hence distinct choices give distinct solutions. In particular, we obtain exactly $2^{\omega(n')}$ different solutions up to multiplication by $\pm\phi^{2h}$; up to conjugation and units, we get a set $\mathcal{F}_n$ comprising $\Omega := 2^{\omega(n')-1}$ solutions.

%
%
%

Notice that, by construction, $\mathcal F_n$ has the property that every coprime solution of \eqref{e:gpell} differs from a solution in $\mathcal{F}_n$ by conjugation and multiplication by units.
\end{rmk}

\begin{defn}\label{d:fund_sols}
We call the set $\mathcal{F}_n$ above a \emph{generating set of solutions} of Equation~\eqref{e:gpell}.
\end{defn}

\subsection{Solving equation~\eqref{eq:pell}}
We now recall that we are actually looking for solutions of Equation~\eqref{eq:pell}
with $\gcd(a,b) = 1$.

\begin{defn}\label{d:correspond}
We say that a pair $(a,b)$ with $a < b$ of positive integers \e{corresponds to} an integer solution $(x,y)$ of the Pell equation $x^2 - 5y^2 = 4(2g-1)$ if $x = (7b - 2a)/3$ and $y = b$. In this case we also say that $(a,b)$ \e{corresponds} to the element $\zeta = x + y\sqrt{5} \in \O$.
\end{defn}

\begin{prop}\label{prop:abcoprime}
Let $\mathcal{F}_{2g-1}$ be a generating set of solutions of Equation~\hyperref[e:gpell]{$(\varspadesuit_{2g-1})$}, in the sense of Definition~\ref{d:fund_sols}.
If $(a,b)$ is a coprime solution of Equation~\eqref{eq:pell}, then $(a,b)$ corresponds to either $\pm2\phi^{2h}\beta$ or $\pm2\phi^{2h}\overline\beta$ for some $\beta\in\mathcal{F}_{2g-1}$.

Conversely, given a solution $\beta\in\mathcal{F}_{2g-1}$:
\begin{itemize}
\item if $g\equiv 0\pmod 3$, then both $\pm2\phi^{2h}\beta$ and $\pm2\phi^{2h}\overline\beta$ correspond to coprime solutions of~\eqref{eq:pell};

\item if $g\equiv 1\pmod3$, then either both $\pm2\phi^{2h}\beta$ and $\pm2\phi^{2h}\overline\beta$ correspond to coprime solutions of~\eqref{eq:pell} for all even values of $h$, or they both do for all odd values of $h$;

\item if $g\equiv 2\pmod3$, $3$ divides $2g-1$, and there are no coprime solutions of~\hyperref[e:gpell]{$(\varspadesuit_{2g-1})$}.
\end{itemize}
\end{prop}

The second half of the statement above can be thought of in the following way. Let $g$ be a positive integer such that $2g-1$ is either $n'$ or $5n'$, where $n'$ is a product of primes congruent to $\pm1$ modulo 5. There are $\Omega=2^{\omega(n')-1}$ families of solutions of Equation~\eqref{eq:pell}, and in each family any two members differ, up to sign and conjugation, by a power of $\phi^2$; if $g\equiv1\pmod 3$, this power is always an even power, i.e.~a power of $\phi^4$.

\begin{proof}
We claim that if $(a,b)$ is a solution of~\eqref{eq:pell} that corresponds to $(x,y)$, then $\gcd(a,b)=1$ if and only if $\gcd(x,y)$ is either 1 or 2, and 3 does not divide $y$.

In fact, if $(a,b)$ satisfies~\eqref{eq:pell} and it corresponds to $(x,y)$, $y=b$ is an integer, and $x^2 = 4(2g-1)+5y^2$, hence $x$ is an integer, too. That is, $3$ divides $3x=7b-2a$. Moreover,
\begin{equation}\label{e:gcds2}
\gcd(x,y) \mid \gcd(3x,y) = \gcd(7b-2a,b) = \gcd(2a,b) \mid 2\gcd(a,b)
\end{equation}
therefore, if $\gcd(a,b)=1$, $3\nmid y$ and then $\gcd(x,y)$ divides $2$.

Conversely, observe that from~\eqref{e:gcds2} it follows that if $3$ divided $y$, then $3$ would divide $\gcd(a,b)$, hence $a$ and $b$ would not be coprime.
Therefore, $\gcd(x,y) = \gcd(3x,y)$.
Moreover:
\begin{itemize}
\item[(i)] if $\gcd(x,y) = 1$, it follows from~\eqref{e:gcds2} that $b$ is odd, and therefore $\gcd(2a,b) = \gcd(a,b) = 1$;

\item[(ii)] if $\gcd(x,y) = 2$, then from~\eqref{e:gcds2} we obtain that $b=y$ is even.
Parity considerations in equation~\hyperref[e:gpell]{$(\varspadesuit_{4(2g-1)})$} show that $x$ is even, too, and hence $(x,y) = (2x_0,2y_0)$, where $x_0^2-5y_0^2 = 2g-1$.
We now have $(a,b) = (\pm(7y_0-3x_0),2y_0)$, from which $\gcd(a,b) = 1$, as desired.
\end{itemize}

We now set out to prove that there exists $\beta\in\mathcal F_{2g-1}$ such that $(a,b)$ corresponds to $2\beta$, up to signs and powers of $\phi^2$.

Let us look at case (i) first.
Given coprime integers $x$, $y$ with $x^2-5y^2 = 4(2g-1)$ such that $3\nmid y$, we obtain the pair $(a,b)$ with $a=\pm(7y-3x)/2$.
As noted above $y=b$ is odd, and parity considerations imply that so is $x$.
Hence $a$ is automatically an integer.
Notice that, since $x$ and $y$ are both odd, $x+y\sqrt5=2\beta'$ for some $\beta'\in\O$.
Similarly as in Claim~\ref{c:integral}, there exists $\beta_1 = x_1+y_1\sqrt 5 \in\Z[\sqrt5]$ such that $(a,b)$ corresponds to $2\phi^{\pm2}\cdot\beta_1$.

In case (ii), we obtained two coprime integers $x_0$, $y_0$ such that $x_0^2-5y_0^2 = 2g-1$; call $\beta_2 = x_0+y_0\sqrt5$

In either case, by definition, $\beta_i$ ($i=1,2$) differs by an element in $\mathcal F_{2g-1}$ by conjugation and a power of $\phi^2$.

To prove the second part of the statement, we need to understand in which cases the second component of $(x,y)$ is not divisible by 3, i.e. when $3|y$ in a coprime solution $(x,y)$ of~\hyperref[e:gpell]{$(\varspadesuit_{4(2g-1)})$}.

Suppose that $x+y\sqrt5$ has norm $N(x+y\sqrt5) \equiv 2\pmod 3$: since 2 is not a quadratic residue modulo 3, then 3 does not divide $y$. On the other hand, suppose that $N(x+y\sqrt5) \equiv 1\pmod 3$: in this case there are solutions with $y\equiv0\pmod3$, and this subset of solutions is acted upon by $\phi^4=\frac{7+3\sqrt5}2$. Likewise, the subset of solutions with $y\not\equiv 0\pmod3$ is acted upon by $\phi^4$, and multiplication by $\phi^2$ takes one family to the other.
\end{proof}

%
%
%
%

\begin{rmk}\label{r:uniqueness}
If we also assume that $2g-1$ is a power of a prime, $\Omega = 1$, hence for any two coprime solutions $(a,b)$, $(a',b')$ of Equation~\eqref{eq:pell} there is an integer $h$ such that the corresponding elements $\zeta, \zeta'\in \O$ satisfy $\zeta = \pm\phi^{2h}\zeta'$ or $\overline\zeta = \pm\phi^{2h}\zeta'$. Moreover, if $g\equiv 1\pmod3$, then $h$ is even.
\end{rmk}

Recall that we are interested in triangular numbers, which correspond to genera of smooth plane curves, and that we are going to construct curves of triangular genus (Theorem~\ref{t:construction}).

\begin{ex}
Let $k$ be an integer, and define $g$ as $g:=k(k-1)/2$ and $n := 2g - 1 = k(k-1)-1$. Notice that, by Proposition~\ref{p:coprime}, $n$ has no prime factors congruent to $\pm2$ modulo 5, since $(2k-1,1)$ is a coprime solution of Equation~\hyperref[e:gpell]{$(\varspadesuit_{4n})$}.

For small triangular numbers, the size $\Omega$ of the generating set of solutions is often 1. In fact, the smallest case in which $\Omega>1$ is $g = 105$, when $2g-1 = 209 = 11\cdot 19$; we then have two integers $\alpha,\beta\in\O$ with $N(\alpha) = 11$, $N(\beta) = 19$, and these give rise to the two solutions $\alpha\cdot\beta$ and $\alpha\cdot\overline\beta$, and these correspond to the two elements of the generating set of solutions of~\hyperref[e:gpell]{$(\varspadesuit_{2g-1})$}.
\end{ex} 
\section{Examples and applications}\label{s:examples}

In this section we will construct examples and prove Theorem~\ref{t:construction}. Unless otherwise stated, curves will be of genus $g\ge 1$.

\subsection{Cremona transformations and the proof of Theorem~\ref{t:construction}}\label{ss:crem}

In this subsection we examine which solutions of the generalized Pell equation~\eqref{eq:pell} are realizable by $1$-unicuspidal curves.

We use a construction due to Orevkov~\cite{Or} (see also~\cite[Proposition 9.12]{BHL}). 
Let $N$ be a nodal cubic; denote by $B_1$ and $B_2$ the two smooth local branches at the node.
We define a birational transformation $f^N_1: \CP^2 \dashrightarrow \CP^2$ as follows. 
Blow up seven points infinitely close to the node of $N$ at branch $B_1$ (resulting in a chain of seven exceptional divisors $E_1, \dots, E_7$, where the divisors are indexed by the order of appearance); then, in the resulting configuration of divisors, blow down the proper transform of $N$ and six more exceptional divisors ($E_1, \dots, E_6$ in this order). 
The birational map $f^N_2$ is defined analogously, only we blow up at points on the branch $B_2$ instead of $B_1$. In both cases, after the last blowdown, the image of the 
exceptional divisor appearing at the last blowup (that is, $E_7$) is a nodal cubic~\cite[Section 6]{Or}: we denote them as $N^1$ and $N^2$ respectively, and we say that they are \e{associated} to $f^N_1$ and $f^N_2$, respectively, and each of them is \e{associated} to $N$.

To state the main technical result of this section, we introduce some terminology.
In this section, we want to allow the concept of a ``smooth Puiseux pair'', by formally allowing pairs $(1,b)$ for $b>1$; notice that this is not a Puiseux pair in the usual sense, since it corresponds to a smooth point on the curve. 
By extension, a smooth curve will be a $(1,b)$-unicuspidal curve for every $b>1$.

\begin{defn}\label{d:sweep}
Let $N$ be a nodal cubic with node $p$, and denote with $B_1$, $B_2$ the two branches of $N$ at the point $p$. We say that an $(a,b)$-unicuspidal curve $C$ of degree $d$ \emph{sweeps} $N$ if:
\begin{enumerate}
\item the cusp of $C$ is at $p$;
\item the branch of $C$ at $p$ has intersection multiplicity $a$ with $B_1$ and $b$ with $B_2$;
\item the only intersection point of $C$ and $N$ is $p$.
\end{enumerate}
\end{defn}

Notice that B\'ezout's theorem implies that, assuming (1) and (2) in the definition above, (3) is equivalent to the condition $a+b = 3d$; when $C$ is smooth, i.e. when $a=1$, this determines $b=3d-1$.

Finally, observe that $C$ defines an ordering of the two branches of $N$ at the node, where the first branch is the one with the lower multiplicity of intersection with $C$.

Now we set up some notation and recall some classical facts about how Puiseux pairs behave under blowing up and blowing down.

Given two curves $C_1$ and $C_2$ in a surface $X$ and a point $r\in X$, denote with $(C_1 \cdot C_2)_r$ the local intersection multiplicity of $C_1$ and $C_2$ at $r$. We use the same notation when $C_1$ and $C_2$ are just local curve branches rather than curves.

\begin{lemma}[{\cite[Theorem 3.5.5]{Wall},~\cite[Proposition 3.1]{Or}}]
Let $\sigma: X \rightarrow \C^2$ be the blowup of $\C^2$ at a point $q$ with exceptional divisor $E$; let $C$ and $B$ be a local irreducible singular branch and a smooth local branch at $q$ respectively, with strict transforms $\widetilde{C}$ and $\widetilde{B}$; let $T$ be a smooth curve branch in $X$ intersecting $E$ transversely. Denote with $p$ the intersection point of $\widetilde{C}$ with $E$.

If $C$ has one Puiseux pair $(a,b)$ at $q$, then $\widetilde{C}$ has a singularity at $p$ of type:
\begin{itemize}
\item $(a, b-a)$, if $b \geq 2a$;
\item $(b-a, a)$, if $b < 2a$.
\end{itemize}
In both cases, $(E \cdot \widetilde{C})_p = a$ and $(\widetilde{B} \cdot \widetilde{C})_p = (B \cdot C)_q - a$.

If $\widetilde{C}$ has one Puiseux pair $(a,b)$ at $p$, then $C$ has a singularity at $q$ of the following type:
\begin{itemize}
\item one Puiseux pair $(a, a+b)$, if $(E \cdot \widetilde{C})_p = a$;
\item one Puiseux pair $(b, a+b)$, if $(E \cdot \widetilde{C})_p = b$;
\item two Puiseux pairs, if $a < (E \cdot \widetilde{C})_p < b$.
\end{itemize}
In any case, $(T \cdot \widetilde{C})_p + (E \cdot \widetilde{C})_p = (\sigma(T) \cdot C)_q$.

When $\widetilde{C}$ is smooth (that is, when $a = 1$), then in case of $(E\cdot\widetilde{C})_p = a = 1$, $C$ is also smooth (consistently with the formal fact that its ``Puiseux pair'' begins with $a=1$); and in case of $(E\cdot\widetilde{C})_p = b > 1$, $C$ is singular and has singularity type $(b, a+b) = (b, b+1)$.
\end{lemma}

The main technical result of this section is the following:

\begin{prop}\label{prop:orevkovtrick}
Let $C$ be an $(a,b)$-unicuspidal curve of genus $g$ and degree $d = (a+b)/3$ that sweeps the nodal cubic $N$; suppose that $(a,b)$ corresponds to $\zeta = x + y\sqrt{5}\in\O$ (in the sense of Definition~\ref{d:correspond}). Denote with $f^N_1, f^N_2$ the rational transformations associated to $N$, and let $N^1$ and $N^2$ be the associated nodal cubics, as described above; here we require that the labelling of the branches is the one induced by $C$. Then:
\begin{enumerate}
\item $f^N_1(C)$ is a $(b, 7b-a)$-unicuspidal curve of genus $g$ sweeping $N^1$; moreover, $(b,7b-a)$ corresponds to $\zeta \phi^4$;

\item if $b < 7a$, $f^N_2(C)$ is a $(7a-b, a)$-unicuspidal curve of genus $g$ sweeping $N^2$; moreover, $(7a - b, a)$ corresponds to $\zeta \phi^{-4}$;

\item if $b > 7a$, $f^N_2(C)$ is a $(b-7a, 7b-48a)$-unicuspidal curve of genus $g$ sweeping $N^2$; moreover, $(b - 7a, 7b - 48a)$ corresponds to $\overline{\zeta \phi^{-12}}$.
\end{enumerate}
\end{prop}

\begin{proof}
We first observe that the genus is invariant under birational transformations, hence $f_i^N(C)$ has genus $g$ for $i=1,2$.

Using the previous lemma, it is not hard to follow what happens with a Puiseux pair $(a,b)$ under the blowups and blowdowns giving the birational transformations $f^N_1$, $f^N_2$.

Notice that we can assume the inequality $2a < b$ when we consider $f^N_1$ and inequality $13a < 2b$ when we consider $f^N_2$ (these are needed to conveniently analyse the occurring blowups and blowdowns).
We will argue that $b/a>\phi^4$; this is enough, since $\phi^4$ is larger than both $13/2$ and $2$.
In fact, $a + b = 3d$ because $C$ sweeps $N$; moreover, $(a, b)$ is a solution of the degree-genus formula~\eqref{eq:dg_one_p}, hence it also satisfies~\eqref{eq:pell}. In particular, the ratio $b/a$ is larger than $\phi^4$, as desired; this is represented in Figure~\ref{f:gammas}: $(a,b)$ belongs to the hyperbola $\gamma_0$, which lies above its asymptote $b/a = \phi^4$. Notice that here we are using the assumption $g\ge 1$, since in the rational case the hyperbola $\gamma_0$ lies \emph{below} its asymptote. (See Section~\ref{ss:3d} and Remark~\ref{rmk:g02ndprop}.)

Notice also that in the third case, i.e.~when applying $f^N_2$ to an $(a,b)$-unicuspidal curve with $b > 7a$, the proof differs slightly according to whether $b < 8a$ or $b \geq 8a$; in both cases, however, the final outcome is $(b - 7a, 7b - 48a)$.

For example, the proof of case (1) goes as follows. (For simplicity, through the whole series of blowups and blowdowns, we will use the same notation for a curve/branch and for its strict transform in the other surface.)
As $C$ with a cusp type $(a,b)$ sweeps the nodal cubic $N$ at the beginning, we have local intersection multiplicities at the node $(C\cdot B_1) = a$, $(C\cdot B_2) = b$. Thus, by Lemma 8.2, first bullet point, after the first blowup, $C$ has a cusp type $(a,b-a)$ (recall that we can assume $2a \leq b$), $(C\cdot E_1) = a$ and $(C\cdot B_2) = b-a$. The next $6$ blowups do not affect the cusp type and these local intersection multiplicities, as everything happens at the other branch $B_1$. Now when blowing down $N$, by Lemma 8.2, we get a singularity of type $(b-a,b)$, and local intersections $(C\cdot E_1) = a + b - a = b$, $(C\cdot E_7) = b-a$. Then, after blowing down $E_i$, $i=1\dots 6$, we have a cusp type $(b,(i+1)b-a)$ and local intersections $(C\cdot E_{i+1}) = b$, $(C\cdot E_7) = (i+1)b-a$. In particular, after the last blowdown, the strict transform of $E_7$ is a nodal cubic with node at the singular point of $C$ such that $C$ has intersection multiplicity $b$ with one branch and $7b-a$ with the other. We will call the first branch $B_1$, the second $B_2$ and (the strict transform of) the divisor $E_7$ we denote by $N^1$. We see that (the strict transform of) $C$ now sweeps the nodal cubic $N^1$ as claimed.

The proof of the remaining two cases is a similar step-by-step check of cusp types and local intersection multiplicities.
\end{proof}

For elements $\zeta = x + y\sqrt{5}$ of $\mathbb{Z}[\sqrt5]$ introduce the notation $\zeta^+ = x + |y|\sqrt{5}$. We have the following corollary of the proposition above.

\begin{lemma}\label{cor:allbuttwo}
Let $C$ be an $(a,b)$-unicuspidal curve of genus $g$ sweeping a nodal cubic $N$, and suppose that $(a,b)$ corresponds to $\zeta\in\mathbb{Z}[\sqrt5]$.
Then there exists $m$ such that for all $i\neq m-1,m$ there exists a $1$-unicuspidal curve of genus $g$ with the Puiseux pair corresponding to $(\zeta \phi^{4i})^+$.
Finally, if $b>7a$, we can choose $m=-1$.
\end{lemma}

\begin{proof}
Suppose for simplicity that $b>7a$; we will see later that this assumption is not very restrictive.
For each $i\ge0$ we are going to construct an $(a_i,b_i)$-unicuspidal curve $C_i$ and an $(a'_i, b'_i)$-unicuspidal curve $\overline C_i$, where $(a_i,b_i)$ corresponds to $(\zeta \phi^{4i})^+$ and $(a'_i,b'_i)$ corresponds to $(\zeta \phi^{-4(i+3)})^+$.

The first family of curves is constructed as follows.
Let $C_0 = C$, $N_0 = N$. For every $i\ge0$, define inductively $C_{i+1}= f^{N_i}_1(C_i)$ and let $N_{i+1} = N_i^1$ be the nodal cubic associated to $f_1^{N_i}$.
Then by Proposition~\ref{prop:orevkovtrick}(1), $C_i$ is an $(a_i, b_i)$-unicuspidal curve of genus $g$ with $(a_i,b_i)$ corresponding to $\phi^{4i}\zeta$.

The second family is constructed in an analogous way.
Define $\overline C_0 = f_2^{N}(C)$, $\overline N_0 = N^2$.
By Proposition~\ref{prop:orevkovtrick}(3), $\overline C_0$ is a $(b-7a, 7b-48a)$-unicuspidal curve.
Call $(a'_0,b'_0)$ the Puiseux pair $(b-7a, 7b-48a)$, which corresponds to $\overline{\zeta\phi^{-12}}$.

As above, for $i \geq 0$, define inductively $\overline{C}_{i+1} = f^{\overline N_i}_1(\overline{C}_{i})$ and let $\overline{N}_{i+1}=\overline N\vphantom{N_i}^1_i$ be the nodal curve associated to ${\overline{N}_i}$.
We remark here that we use the labelling of the branches of $\overline{N}_{i}$ induced by the curve $\overline{C}_i$.

Then, by Proposition~\ref{prop:orevkovtrick}(1), $\overline C_i$ is an $(a'_i, b'_i)$-unicuspidal curve of genus $g$, where $(a'_i,b'_i)$ corresponds to $\phi^{4i}\cdot\overline{\zeta\phi^{-12}}$; however, observe that
\[
\overline{\zeta\phi^{-12}}\cdot \phi^{4i} = \overline{\zeta\phi^{-12}}\cdot \overline{\phi^{-4i}} = \overline{\zeta\phi^{-4(i+3)}} = (\zeta\phi^{-4(i+3)})^{+}.
\]

This concludes the proof under the assumption $b>7a$; we prove now that this assumption is not needed.
To this end, assume that $b<7a$, and consider the curve $f_2^N(C)$; by Proposition~\ref{prop:orevkovtrick}(2), this is an $(a^*,b^*)$-unicuspidal curve of genus $g$, and $(a^*,b^*) = (7a-b,a)$ corresponds to $\phi^{-4}\zeta$.

As remarked in the proof of Proposition~\ref{prop:orevkovtrick}, the assumption $a+b=3d$ implies that $b/a > \phi^4> 41/6$, hence it is easy to see that $b^*-7a^* > b - 7a$.
An inductive argument shows that there exists an integer $m\le 0$ such that there exists an $(a^*,b^*)$-unicuspidal curve $C^*$ with $b^*>7a^*$, with $(a^*, b^*)$ corresponding to $\phi^{4(m+1)}\zeta$, that sweeps a nodal curve $N^*$.
That is, the assumption $b>7a$ was not restrictive, and this concludes the proof.
\end{proof}

Before turning to the proof of Theorem~\ref{t:construction}, we prove the following:

\begin{lemma}\label{lem:smoothatnode}
Let $N \subset \CP^2$ be the nodal cubic defined by the equation $x^3+y^3-xyz = 0$. Then for any positive integer $d$, there exists a reduced smooth projective curve $C$ of degree $d$ having local intersection multiplicity $3d$ with $N$ at its node.
\end{lemma}

\begin{proof}
Set $h(x,y,z) = x^3+y^3-xyz$ for the polynomial defining $N$. The node of $N$ is at the point $[x:y:z] = [0:0:1]$.

First, we prove that there exists a smooth \emph{local} curve branch of degree at most $d$ having local intersection multiplicity $3d$ with $N$ at its node. To show this, we will work in the affine chart $z = 1$ of $\CP^2$. $N \subset \CP^2$ has a parametrisation with $[t:s] \in \CP^1$ as follows:
\[ [x:y:z] = [ts^2: t^2s : t^3+s^3].\]
Therefore, in the affine chart with coordinates $x$ and $y$ around $(0,0)$,
\[(x(t),y(t)) = \left(\sum_{k=0}^{\infty}(-1)^kt^{3k+1}, \sum_{k=0}^{\infty}(-1)^kt^{3k+2}\right), \quad |t|<\frac12\]
is a parametrisation of the branch of $N$ tangent to the $x$-axis, mapping $t=0$ to the node. For simplicity, while working in this chart, we denote by $h(x,y)$ the polynomial $x^3+y^3-xy$ as well.

We claim that for each $d$ there exists a polynomial $f_d(x,y)$ of degree $d$ such that it has local intersection multiplicity $3d$ with $N$ at its node and it is smooth at that point. For $d=1,2$ we set
\[ f_1(x,y) = y,\quad f_2(x,y) = y-x^2: \]
the order of $t$ in the expansion of $f_1(x(t),y(t))$ is $2$, and the coefficient of $t^2$ is $1$; and the order of $t$ in the expansion of $f_2(x(t),y(t))$ is $5$, and the coefficient of $t^5$ is $1$.

Let $c_1 = c_2 = 1$. For $n \geq 3$, define recursively the pair $(f_n(x,y), c_n)$ as follows:
\[ f_n(x,y) = c_{n-2}f_{n-1}(x,y) - c_{n-1}xyf_{n-2}(x,y) \]
and let $c_n$ denote the coefficient of $t^{3n-1}$ in the expansion of $f_n(x(t),y(t))$.

\begin{claim}
For any integer $n \geq 1$, the following hold:

\begin{enumerate}
\item[(1)] $f_n(x,y)$ is a polynomial of degree $n$. If the coefficient of $x^iy^j$ in $f_n(x,y)$ is nonzero, then $i+2j \equiv 2 \pmod 3$. The coefficient of the monomial $y$ in $f_n(x,y)$ is nonzero.

\item[(2)] The order of $t$ in the expansion of $f_n(x(t),y(t))$ is $3n-1$; in particular, $c_n \neq 0$.
\end{enumerate}
\end{claim}

\begin{proof}
We prove this by induction on $n$. Both properties can be easily checked for $n=1,2$. Assume we know the claim for $f_{n-1}(x,y)$ and $f_{n-2}(x,y)$, and let us prove it for $f_n(x,y)$. For an integer $m$, let $Q_m(t)$ be defined by $f_m(x(t),y(t))/t^{3m-1}$; the inductive hypothesis tells us that $Q_{n-1}(t)$ and $Q_{n-2}(t)$ are power series satisfying $Q_{n-1}(0) = c_{n-1} \neq 0$ and $Q_{n-2}(0) = c_{n-2} \neq 0$.

All the properties in (1) are obvious from the definition; for the non-vanishing of the coefficient of $y$, we observe that it is obtained from the coefficient of $f_{n-1}(x,y)$ by multiplication with $c_{n-2} \neq 0$, and therefore it is nonzero.

For (2), write
\[ f_n(x(t),y(t)) = \sum_{m=0}^{\infty}a_mt^m. \]
First, we prove that $a_m=0$ for $m \leq 3n-2$. As all powers of $t$ in the expansion of $x(t)$ are congruent to $1 \pmod 3$ and all powers of $t$ in the expansion of $y(t)$ are congruent to $2 \pmod 3$, using part (1) we immediately get that $a_m = 0$ if $m \not\equiv 2 \pmod 3$. Therefore, it is enough to show that $a_m = 0$ for $m\equiv2\pmod3$ with $m \leq 3n-4$. From the definition of $f_n(x,y)$, $Q_{n-1}(t)$, $Q_{n-2}(t)$, and the inductive hypothesis, we have
\[ f_n(x(t),y(t)) = c_{n-2}t^{3(n-1)-1}Q_{n-1}(t) - c_{n-1}t^3P(t)\cdot t^{3(n-2)-1}Q_{n-2}(t)  \]
where $P(t)$ is the power series associated to $(1+t^3)^{-2}$. Therefore, there is indeed no power $t^m$ in the above expansion for $m \leq 3n-4$.

Assume now that $t^{3n-1}$ vanishes as well, i.e.~$c_n = 0$. This would mean that $f_n(x,y)$ has local intersection multiplicity at least $3n$ with the parametrized branch of $N$, therefore (since it also intersects the other branch at the node), it has intersection multiplicity at least $3n+1$ with the cubic $N$ altogether. As the degree of $f_n(x,y)$ is $n$, from B\'{e}zout's theorem, it follows that $h(x,y)$ divides $f_n(x,y)$. But this is impossible, as $y$ has non-zero coefficient in $f_n(x,y)$; a contradiction.
\end{proof}

\begin{rmk}
Numerical evidence indicates that $c_n$ is in fact $1$ for every $n \geq 1$. However, this is not needed in the argument.
\end{rmk}

The curve $C_d$ defined by the equation $f_d(x,y)=0$ has local intersection multiplicity $3d-1$ with one branch of the node of $N$, and at least one with the other branch. Since it does not contain the curve $N$, its intersection multiplicity with $N$ at the node is at most $3d$, by B\'ezout's theorem. This also shows that the intersection multiplicity of $C_d$ with one of the two branches of $N$ is 1, hence $C_d$ is reduced.

In this way, we have proven the existence of a local germ $f_d(x,y)$ with the above properties. We will also denote with $f_d(x,y,z)$ the homogenisation of the germ, and with $C_d$ the associated, degree-$d$ projective curve. As indicated above, $C_d$ is reduced and it intersects $N$ only at the node with multiplicity $3d$. However, it may have singular points away from $[0:0:1]$ if $d \geq 3$.

Recall that we have defined two polynomials $f_1(x,y)$ and $f_2(x,y)$ of degrees 1 and 2 respectively, whose zero sets are curves that sweep $N$ at the node. Suppose now that $d>2$. For a non-negative integer $m$ let $g_m(x,y,z)$ be the Fermat polynomial $g_m(x,y,z) = x^m + y^m + z^m$, defining the Fermat curve $\{g_m(x,y,z) = 0\}$.

Consider the linear pencil of curves
\[  \{\lambda f_d(x,y,z) + \mu h(x,y,z)g_{d-3}(x,y,z) = 0\} \]
as $[\lambda: \mu]$ varies in $\CP^1$.

The basepoints of the linear system are the intersection points of $C_d$ and $\{h(x,y,z)g_{d-3}(x,y,z) = 0\}$. These are the point $[0:0:1]$ and the intersections of $\{g_{d-3}(x,y,z) = 0\}$ with $C_d$.

We claim that at each basepoint one of the two generating curves of the pencil is smooth. In fact, at $[0:0:1]$ the curve $C_d$ is smooth by construction; on the other hand, since $C_d$ and $N$ intersect only at $[0:0:1]$, at every other basepoint the curve $\{h(x,y,z)g_{d-3}(x,y,z)=0\}$ coincides with the Fermat curve, and hence it is smooth. Therefore, the generic member of the pencil is smooth and reduced by Bertini's theorem~\cite[Theorem 17.16]{Harris}.

Notice also that the intersection multiplicity of both generators of the pencil with $N$ is at least $3d$: it is at least $3d-1$ with one branch $B_2$ and at least $1$ with the other branch $B_1$ of $N$ at the node. In this way, for all members of the pencil the intersection multiplicity is at least $3d$ (at least $3d-1$ with $B_2$ and at least $1$ with $B_1$). Now again by B\'ezout's theorem, it can not be higher. In particular, the generic member of the pencil has intersection multiplicity exactly $3d-1$ with $B_2$ and $1$ with $B_1$; altogether $3d$ with $N$ at its node.  So the generic member of the pencil is the curve $C$ we are looking for.
\end{proof}

\begin{proof}[Proof of Theorem~\ref{t:construction}]
Recall from Proposition~\ref{prop:orevkovtrick} that a curve $C$ sweeping a cubic $N$ induces a labelling of the branches of $N$ at its node, hence the pair $(C,N)$ determines two Cremona transformations $f^{N}_1$ and $f^N_2$ which in turn define the two associated nodal cubics $N^1$ and $N^2$. Moreover, $f^N_j(C)$ sweeps $N^j$ for $j=1,2$.

By Lemma~\ref{lem:smoothatnode} we can find a smooth curve $C$ of degree $d=k+1=L^k_3$ and genus $g = k(k-1)/2$ that sweeps a nodal cubic $N$.
Notice that, since $C$ is smooth, the `Puiseux pair' of the singularity at the node is $(L^k_1, L^k_5) = (1,3k+2)$; also, observe that $(1,3k+2)$ corresponds to $\zeta = 7k+4+(3k+2)\sqrt5$.

Also, since $k>1$, $L^k_5 > 7L^k_1$, Lemma~\ref{cor:allbuttwo} gives, for every $i\neq -1,-2$, a 1-unicuspidal curve whose Puiseux pair corresponds to $(\phi^{4i}\zeta)^+$.
Now, simply observe that $(L^k_{4i+1}, L^k_{4i+5})$ corresponds to $\phi^{4i}\zeta$. 

As for the second part of the stament, by Theorem~\ref{thm:pell}, almost all Puiseux pairs $(a,b)$ such that there exists an $(a,b)$-unicuspidal curve of genus $g$ satisfy~\eqref{eq:pell}.
As noted in Remark~\ref{r:uniqueness}, if $g\equiv1\pmod 3$ and $2g-1$ is a power of a prime, then all solutions of the generalised Pell equation~\eqref{eq:pell} correspond to elements of form $(\zeta\phi^{4n})^+$, $n\in \mathbb{Z}$ for some $\zeta \in \mathcal{O}$. 
On the other hand, since we realised all but two elements in a $\pm\phi^4$-orbit, for these genera we have constructed examples corresponding to \e{almost all} singularities of 1-unicuspidal curves.
\end{proof}

\begin{ex}
Let us consider the case $g=6$. According to Theorem~\ref{t:construction}, for every $n\neq -1,-2$ we constructed a singular genus-6 curve with a cusp of type corresponding to $(\phi^{4n}(32 + 14\sqrt{5}))^+ \in \O$.

Notice that, since $g$ is divisible by 3, by Proposition~\ref{prop:abcoprime}, the elements $(\phi^{4n+2}(32 + 14\sqrt{5}))^+$, too, correspond to solutions of Equation~\eqref{eq:pell}.
\end{ex}

\begin{ques}\label{q:g6}
Is it the case that for almost all $n$ the elements
\[(\phi^{4n+2}(32 + 14\sqrt{5}))^+\in\O\]
correspond to the Puiseux pair of a 1-unicuspidal curve $C_n$?
\end{ques}

In light of Lemma~\ref{cor:allbuttwo}, the answer is positive if we can find an 
integer $n^*$ and an $(a^*, b^*)$-unicuspidal curve $C^*$ that sweeps a nodal cubic $N^*$, where $(a^*, b^*)$ corresponds to $(\phi^{4n^*+2}(32 + 14\sqrt{5}))^+$.

Notice also that, since $2g-1 = 11$ is a prime, together with the previous construction, this would realise almost all possible Puiseux pairs for this genus (see Remark~\ref{r:uniqueness}).

\begin{ex}
Let us consider the case $g=10$.
Theorem~\ref{t:construction} gives us a family of curves, each with a singularity corresponding to $(\phi^{4n}(39 + 17\sqrt{5}))^+ \in \O$ for any $n\neq-1,-2$.

Since $g\equiv 1\pmod 3$ and $2g-1 = 19$ is a prime, by Remark~\ref{r:uniqueness} all the solutions of the Pell equation which can correspond to possible Puiseux pairs are in the set $\{(\zeta \phi^{4n})^{+} : n \in \mathbb{Z}\}$.

So we constructed all but finitely many $1$-unicuspidal genus-$10$ curves (up to equisingularity).
\end{ex}

Similar arguments apply as soon as $g$ is a triangular number such that $2g-1$ is a power of a prime greater than $5$.
That is, if $g\equiv1\pmod3$, then Theorem~\ref{t:construction} produces almost all 1-unicuspidal curves of genus $g$, up to equisingularity.
The first triangular number for which we are unable to produce almost all 1-unicuspidal curves is $g=325$, for which $2g-1 = 649 = 11\cdot 59$.

If, on the other hand, $2g-1$ is congruent to $-1$ modulo $3$ and larger than 5, Theorem~\ref{t:construction} only provides explicit examples of 1-unicuspidal curves corresponding to \e{half} of the solutions of Equation~\eqref{eq:pell}. We have a natural generalisation of Question~\ref{q:g6} above:

\begin{ques}
Let $g>3$ be a triangular number such that $g\equiv0\pmod3$ and $2g-1$ is a power of a prime. Is it the case that for almost all coprime solutions $(a,b)$ of Equation~\eqref{eq:pell} there exists an $(a,b)$-unicuspidal curve of genus $g$?
\end{ques}

\begin{ex}\label{ex:g1_3}
For $g = 1$ and $g = 3$ some unusual things happen, as $2g-1$ is 1 in the first case, and 5 in the second.

If $g = 1$, we can start with $(a,b) = (1,8)$ corresponding to $\zeta = 18+8\sqrt{5}$. We get that there exist curves corresponding to $(\zeta \phi^{4n})^+$ for any $n \in \mathbb{Z}$, except $n = -1$ (yielding $\zeta \phi^{-4} = 3 + \sqrt{5}$) and $n = -2$ (yielding $\zeta \phi^{-8} = 3 - \sqrt{5}$). However, now $(\zeta \phi^{4n})^+ = (\zeta \phi^{-4(n+3)})^+$, so the set of possible Puiseux pairs is indexed by $\mathbb{N}$ rather than $\mathbb{Z}$.

If $g = 3$, we can start with $(a,b) = (1,11)$ corresponding to $\zeta = 25+11\sqrt{5}$. We get that there exist curves corresponding to $(\zeta \phi^{4n})^+$ for any $n \in \mathbb{Z}$, except $n = -1$ (yielding $\zeta \phi^{-4} = 5 + \sqrt{5}$) and $n = -2$ (yielding $\zeta \phi^{-8} = 10 - 4\sqrt{5}$). Now although $5 \equiv -1\pmod3$ and $\zeta \phi^2 = 65 + 29\sqrt{5}$ (so $3$ does not divide $y$ here), we do not get another family of solutions, as $\zeta \phi^2 = 65 + 29\sqrt{5} = \overline{\zeta \phi^{-12}}$.
\end{ex}

\begin{ex}
The smallest non-triangular $g$ for which we have infinitely many coprime solutions of the Pell equation is $g = 16$. By Proposition~\ref{prop:abcoprime}, all coprime solutions of Equation~\eqref{eq:pell} correspond to $(\phi^{4n}(57+25\sqrt{5}))^+$, $n \in \mathbb{Z}$.
\end{ex}

The natural question to ask in this setting is the following:

\begin{ques}
Is it the case that for every $g$, almost all coprime solutions $(a,b)$ of Equation~\eqref{eq:pell} are realised by an $(a,b)$-unicuspidal curve of genus $g$?
\end{ques}

\subsection{More examples}
In the next two examples, we exhibit 1-unicuspidal genus-$g$ curves for each genus $g\ge2$ (except $g=3$). These curves are of minimal degree $d_{min}$ among all curves with the given singularity, but $d_{min}$ is not the minimal degree for which the degree-genus formula~\eqref{eq:dg_one_p} has a solution.

Also, in both cases, $a+b$ is not $3d$, thus yielding examples of curves (for infinitely many $g$) not covered by the identity of Theorem~\ref{thm:pell}.

\begin{ex}\label{ex:except1}
Consider the projective curve $C_p$ defined by the equation
\[ C_p = \{ x^{p+3} + y^{p+3} + x^pz^3  = 0 \}\]
for $p \geq 2$, and $p$ not divisible by $3$. This is a unicuspidal curve of degree $d = p+3$ and genus $g = p+2$. The cusp is of type $(p, p+3)$. This local type is always a candidate with $d = p+2$ as well; in this case $g=1$, but it is not an admissible candidate since Equation~\eqref{eq:sem_ineq} obstructs the existence of such curves: set $j = 1$ and $k = 0$ and notice that $R_{d+1} = R_{p+3} = \#(\Gamma\cap[0, p+3)) = 2 < 3$, as the semigroup is generated by two elements $p$ and $p+3$.

Notice that in this case $a+b = 2p+3$, while $3d = 3(p+3)$.
\end{ex}

In fact, for certain local types, Theorem~\ref{thm:main} can exclude arbitrarily many candidate degrees as well, as the next example shows.

\begin{ex}\label{ex:except2}
Consider the projective curve $D_p$ defined by
\[ D_p = \{ x^p z^{p-1} + x^{2p-1} + y^{2p-1} = 0 \}\]
for $p \geq 2$. This is a unicuspidal curve of degree $d = 2p-1$ and genus $g = (p-1)(p-2)$. The cusp has a torus knot of type $(p, 2p-1)$. This local type is always a candidate with all possible smaller $d$'s (such that $\binom{d-1}{2} \geq \delta = (p-1)^2$), but the inequality~\eqref{eq:sem_ineq} obstructs the existence of such curves (set $j = 1$ and $k = 0$ and notice that in order to have $R_{d+1} \geq 3$ we need $d \geq 2p-1$). In this way, this example shows that for the topological type $(p, 2p-1)$ approximately $(2-\sqrt{2})p$ candidate degrees can be obstructed.

In this case, $a+b = 3p-1$, while $3d=6p-3$.
\end{ex}

\begin{rmk}
In fact, the above applications of~\eqref{eq:sem_ineq} (and, surprisingly, the proof of Proposition~\ref{p:3d} too) use only the lower bound for $R_{jd+1-2k}$ and only with $k=0$. This special case can be obtained by applying B\'ezout's theorem. Indeed, one can repeat the argument of the proof of~\cite[Proposition 2]{BLMN} without any change, and get
\[
\frac{(j+1)(j+2)}{2} \leq R_{jd+1}
\]
for $j = 0, \dots, d-3$.
\end{rmk}

Notice that in all the examples above, $a + b < 3d$. The next example shows that there are infinitely many 1-unicuspidal curves with $a + b > 3d$ as well.

\begin{ex}\label{ex:except3}
For any positive integer $n > 2$, there exists a $(3n,21n+1)$-unicuspidal curve $C_n'$ of degree $d = 8n$ and genus $(n-1)(n-2)/2$. To construct such curves, consider a smooth curve $C_n$ of degree $n$ touching a nodal cubic $N$ in one single point outside its node (with local intersection multiplicity $3n$). By Lemma~\ref{lem:smoothatflex} stated below, such a pair $(N,C_n)$ exists. Choose a branch of $N$ at the node, and let $f^N$ denote the associated Orevkov's Cremona transformation (we do not have a specified order of branches at the node here, so there is no point in distinguishing $f^N_1$ and $f^N_2$), as described in Subsection~\ref{ss:crem}. Then one can check that the $C_n' = f^N(C_n)$ has degree and cusp type as described above, therefore, $a + b = 24n+1 > 24n = 3d$.

Recall that the birational map $f^N$ automatically produces a new nodal cubic $N^1$ as the image of $E_7$ (the exceptional divisor obtained at the last blowup) after the series of blowdowns. Notice that we now can distinguish the two branches of $N^1$: call $B_1$ the branch having intersection multiplicity $3n$ with $C_n'$ and  $B_2$ the one having intersection multiplicity $21n$ with $C_n'$. Now applying the birational map $f^{N^1}_2$ leads back to the original smooth curve $C_n$. On the other hand, applying $f^{N^1}_1$ leads to another cuspidal curve; however, this has more than one Puiseux pair.
\end{ex}

\begin{lemma}\label{lem:smoothatflex}
Let $N \subset \CP^2$ be the nodal cubic defined by the equation $x^3 + x^2y - yz^2 = 0$. Then for any positive integer $d > 2$, there exists a reduced smooth projective curve $C$ of degree $d$ having local intersection multiplicity $3d$ with $N$ at the point $[x:y:z] = [0:0:1]$.
\end{lemma}
\begin{proof}
The proof is similar to the proof of Lemma~\ref{lem:smoothatnode}.
Set $h(x,y,z) = x^3 + x^2y - yz^2$ for the equation of $N$ itself.
In the affine chart of $\CP^2$ given by $z = 1$ with coordinates $x$ and $y$ around $(0,0)$,
\[ (x(t), y(t)) = \left(t, \frac{t^3}{1-t^2}\right), \quad |t| < 1/2 \]
is a parametrisation of the local branch of $N$, mapping $t=0$ to the inflection point $(0,0)$.

Set $f_d(x,y) = y^d + x^3 + x^2y - y$. It is easy to check that $f_d(x,y)$ is smooth at $(0,0)$ and has local intersection multiplicity $3d$ with $N$. We will also denote with $f_d(x,y,z)$ the homogenisation of this germ. Consider the linear pencil
\[ \{\lambda f_d(x,y,z) + \mu h(x,y,z) g_{d-3}(x,y,z)=0\}  \]
where $[\lambda:\mu]\in\CP^1$ and $g_{d-3}(x,y,z)$ is the Fermat polynomial defined above. Similarly as in the proof of Lemma~\ref{lem:smoothatnode}, by Bertini's theorem it follows that the generic member of the pencil determines a smooth curve with the desired properties.
\end{proof}

\begin{rmk}
Call an $(a,b)$-unicuspidal degree-$d$ curve \emph{exceptional} if $a + b \neq 3d$. Theorem~\ref{thm:pell} says that there are at most finitely many exceptional curves for each fixed genus $g$. More precisely, analysing the proof of Theorem~\ref{thm:pell}, one can obtain an upper bound quadratic in $g$ for the degree of exceptional curves, that is, $d \leq O(g^2)$ for all exceptional curves of degree $d$ and genus $g$.

Of course, there can be infinitely many exceptional curves of varying genus $g$. Indeed, the above examples provide us infinitely many such curves: in Example~\ref{ex:except1}, we have $a+b<3d$ and $d=O(g)$, in Example~\ref{ex:except2}, we have $a+b<3d$ again but $d=O(\sqrt{g})$; in Example~\ref{ex:except3}, we have $a+b>3d$ and $d=O(\sqrt{g})$ again.
\end{rmk}

The following natural question arises.

\begin{ques}
Let $d_g$ be the largest degree of an exceptional curve of genus $g$. How fast does $d_g$ grow with $g$? What is the smallest power $g^\mu$ of $g$ such that $d_g = O(g^{\mu+\varepsilon})$ for all $\varepsilon>0$?

The same question can be asked separately for exceptional curves with $a+b>3d$ and $a+b<3d$.
\end{ques}

From the proof of Theorem~\ref{thm:pell} and Example~\ref{ex:except1} it follows that the exponent we are looking for is between $1$ and $2$.

\bibliographystyle{amsplain}
\bibliography{unicuspidal}

\providecommand{\bysame}{\leavevmode\hbox to3em{\hrulefill}\thinspace}
\providecommand{\MR}{\relax\ifhmode\unskip\space\fi MR }
\providecommand{\MRhref}[2]{%
  \href{http://www.ams.org/mathscinet-getitem?mr=#1}{#2}
}
\providecommand{\href}[2]{#2}
\begin{thebibliography}{10}

\bibitem{Barth}
Wolf~P. Barth, Klaus Hulek, Chris A.~M. Peters, and Antonius Van~de Ven,
  \emph{Complex compact surfaces}, vol.~4, Springer Science \& Business Media,
  2004.

\bibitem{BHL}
Maciej Borodzik, Matthew Hedden, and Charles Livingston, \emph{Plane algebraic
  curves of arbitrary genus via {H}eegaard {F}loer homology}, preprint
  available at \href{http://arxiv.org/abs/1409.2111}{arXiv.org:1409.2111},
  2014.

\bibitem{BL}
Maciej Borodzik and Charles Livingston, \emph{Heegaard {F}loer homology and
  rational cuspidal curves}, Forum Math. Sigma \textbf{2} (2014), e28, 23.

\bibitem{BLMN}
Javier Fern{\'a}ndez~de Bobadilla, Ignacio Luengo, Alejandro
  Melle-Hern{\'a}ndez, and Andras N{\'e}methi, \emph{On rational cuspidal
  projective plane curves}, Proceedings of the London Mathematical Society
  \textbf{92} (2006), no.~1, 99--138.

\bibitem{BLMN1}
\bysame, \emph{Classification of rational unicuspidal projective curves whose
  singularities have one {P}uiseux pair}, Real and complex singularities,
  Trends Math., Birkh\"auser, Basel, 2007, pp.~31--45.

\bibitem{Harris}
Joe Harris, \emph{Algebraic geometry: a first course}, vol. 133, Springer
  Science \& Business Media, 2013.

\bibitem{Liu}
Tiankai Liu, \emph{On planar rational cuspidal curves}, Ph.D. thesis,
  Massachussets Institute of Technology, 2014.

\bibitem{Marcus}
Daniel~A. Marcus, \emph{Number fields}, vol.~18, Springer, 1977.

\bibitem{MS}
Takashi Matsuoka and Fumio Sakai, \emph{The degree of rational cuspidal
  curves}, Mathematische Annalen \textbf{285} (1989), no.~2, 233--247.

\bibitem{Milnor}
John~W. Milnor, \emph{Singular points of complex hypersurfaces}, no.~61,
  Princeton University Press, 1968.

\bibitem{NW}
Yi~Ni and Zhongtao Wu, \emph{Cosmetic surgeries on knots in {$S^3$}}, J. Reine
  Angew. Math. \textbf{706} (2015), 1--17.

\bibitem{Or}
Stepan~Yu. Orevkov, \emph{On rational cuspidal curves}, Mathematische Annalen
  \textbf{324} (2002), no.~4, 657--673.

\bibitem{OSAG}
Peter~S. Ozsv{\'a}th and Zolt{\'a}n Szab{\'o}, \emph{Absolutely graded {F}loer
  homologies and intersection forms for four-manifolds with boundary}, Adv.
  Math. \textbf{173} (2003), no.~2, 179--261.

\bibitem{OSKI}
\bysame, \emph{Holomorphic disks and knot invariants}, Advances in Mathematics
  \textbf{186} (2004), no.~1, 58--116.

\bibitem{OSPA}
\bysame, \emph{Holomorphic disks and three-manifold invariants: properties and
  applications}, Annals of Mathematics (2004), 1159--1245.

\bibitem{OS3I}
\bysame, \emph{Holomorphic disks and topological invariants for closed
  three-manifolds}, Ann. of Math. (2) \textbf{159} (2004), no.~3, 1027--1158.

\bibitem{OSIS}
\bysame, \emph{Knot {F}loer homology and integer surgeries}, Algebr. Geom.
  Topol. \textbf{8} (2008), no.~1, 101--153.

\bibitem{JakePhD}
Jacob~A. Rasmussen, \emph{Floer homology and knot complements}, Ph.D. thesis,
  Harvard University, 2003.

\bibitem{Jake}
\bysame, \emph{Lens space surgeries and a conjecture of {G}oda and
  {T}eragaito}, Geom. Topol. \textbf{8} (2004), 1013--1031.

\bibitem{Wall}
C.~T.~C. Wall, \emph{Singular points of plane curves}, vol.~63, Cambridge
  University Press, 2004.

\end{thebibliography}

\end{document}